\documentclass[a4paper,10.9pt]{amsart}

\DeclareRobustCommand{\SkipTocEntry}[5]{}

\usepackage{ textcomp }

\usepackage{shadow}
\usepackage{wasysym}

\usepackage{accents}

\usepackage{marvosym}

\usepackage[T1]{fontenc}
\usepackage{empheq}
\usepackage{relsize}
\usepackage{amsmath}

\usepackage{bm}

\usepackage{upgreek}
\usepackage{ esint }
\usepackage{color}
\usepackage{comment}
\usepackage{amssymb}
\usepackage{amsfonts}
\usepackage{graphicx}

\usepackage{amscd}

\usepackage{slashed}
\usepackage{pdflscape}
\usepackage{tikz}
\usepackage{enumerate}
\usepackage{multirow}
\usepackage{stmaryrd}
\usepackage{cancel}

\usepackage{scalerel,stackengine}

\usepackage{ifthen}

\usepackage{bbold}
\usepackage[linkcolor=blue]{hyperref}
\usepackage{mathrsfs}
\usepackage[colorinlistoftodos]{todonotes}

\definecolor{blue}{rgb}{.255,.41,.884} 
\definecolor{red}{rgb}{1, 0, 0} 
\definecolor{green}{rgb}{.196,.804,.196} 
\definecolor{yellow}{rgb}{1,.648,0} 
\definecolor{pink}{rgb}{1,0.5,0.5}

\newtheorem{theorem}{Theorem}[section]
\newtheorem{lemma}[theorem]{Lemma}
  
\newtheorem{proposition}[theorem]{Proposition}

\theoremstyle{definition}
\newtheorem{definition}[theorem]{Definition}
\newtheorem{example}[theorem]{Example}

\theoremstyle{remark}
\newtheorem{remark}[theorem]{Remark}

\usepackage{pdfsync}

\usepackage{graphicx} 
\usepackage{amsmath} 
\usepackage{amsfonts}
\usepackage{amssymb}

\newcommand{\be}{\begin{equation}}

\newcommand{\ee}{\end{equation}}

\newcommand{\II}{{\rm  I\hspace{-.2mm}I}}
\newcommand{\IIo}{\hspace{0.4mm}\mathring{\rm{ I\hspace{-.2mm} I}}{\hspace{.0mm}}}

\newcommand{\IVo}{{\mathring{{\bf\rm I\hspace{-.2mm} V}}{\hspace{.2mm}}}{}}
\newcommand{\Vo}{{\mathring{{\bf\rm V}}}{}}
\newcommand{\VIo}{{\mathring{{\bf\rm V\hspace{-.2mm}I}}{\hspace{.2mm}}}{}}

\newcommand{\otop}{\mathring{\top}}

\newcommand{\si}{\sigma}

\newcommand{\ba}{\begin{array}}

\newcommand{\ea}{\end{array}}

\newcommand{\beq}{\begin{eqnarray}}

\newcommand{\eeq}{\end{eqnarray}}

\newtheorem{lm}{lemma}

\newtheorem{thee}{theorem}

\newtheorem{proo}{proposition}

\newtheorem{co}{corollary}

\newtheorem{rem}{remark}

\newtheorem{deff}{definition}

\newcommand{\bd}{\begin{deff}}

\newcommand{\ed}{\end{deff}}

\newcommand{\bl}{\begin{lm}}

\newcommand{\el}{\end{lm}}

\newcommand{\bp}{\begin{proo}}

\newcommand{\ep}{\end{proo}}

\newcommand{\bt}{\begin{thee}}

\newcommand{\et}{\end{thee}}

\newcommand{\bc}{\begin{co}}

\newcommand{\ec}{\end{co}}

\newcommand{\brm}{\begin{rem}}

\newcommand{\erm}{\end{rem}}

\usepackage{t1enc}

\def\Cal{\mathcal}

\newcommand{\bS}{\mathbb{S}}

\newcommand{\newc}{\newcommand}

\let\ccdot.

\newc{\aR}{\mbox{\boldmath{$ R$}}}
\newc{\aS}{\mbox{\boldmath{$ S$}}}
\newc{\aT}{\mbox{\boldmath{$ T$}}}
\newc{\aW}{\mbox{\boldmath{$ W$}}}

\newc{\aD}{\mbox{\boldmath{$ D$}}\hspace{-.2mm}}

\renewcommand{\colon}{\scalebox{1.2}{:}}

\newc{\aK}{\mbox{\boldmath{$ K$}}}
\newc{\aL}{\mbox{\boldmath{$ L$}}}

\newcommand{\ce}{{\Cal E}}

\newcommand{\ct}{{\Cal T}}

\newcommand{\bT}{{\Bbb T}}

\usepackage{amssymb}
\usepackage{amscd}

\newcommand{\nn}[1]{(\ref{#1})}

\newcommand{\bg}{\mbox{\boldmath{$ g$}}}

\newc{\obstrn}[2]{B^{#1}_{#2}}

\newcommand{\rpl}                         
{\mbox{$
\begin{picture}(12.7,8)(-.5,-1)
\put(0,0.2){$+$}
\put(4.2,2.8){\oval(8,8)[r]}
\end{picture}$}}

\newcommand{\lpl}                         
{\mbox{$
\begin{picture}(12.7,8)(-.5,-1)
\put(2,0.2){$+$}
\put(6.2,2.8){\oval(8,8)[l]}
\end{picture}$}}

\newc{\tensor}[1]{#1}
\newc{\Mvariable}[1]{\mbox{#1}}
\newc{\down}[1]{{}_{#1}}
\newc{\up}[1]{{}^{#1}}

\newc{\JulyStrut}{\rule{0mm}{6mm}}
\newc{\midtenPan}{\mbox{\sf S}}
\newc{\midten}{\mbox{\sf T}}
\newc{\midtenEi}{\mbox{\sf U}}
\newc{\ATen}{\mbox{\sf E}}
\newc{\BTen}{\mbox{\sf F}}
\newc{\CTen}{\mbox{\sf G}}

\def\sideremark#1{\ifvmode\leavevmode\fi\vadjust{\vbox to0pt{\vss
 \hbox to 0pt{\hskip\hsize\hskip1em
 \vbox{\hsize2cm\tiny\raggedright\pretolerance10000
  \noindent #1\hfill}\hss}\vbox to8pt{\vfil}\vss}}}

\numberwithin{equation}{section}

\newcommand{\hh}{{\hspace{.3mm}}}
\newcommand{\mm}{{\hspace{-.2mm}}}

\renewcommand\colon{\scalebox{1.3}{$:$}}

\newcommand{\cc}{\boldsymbol{c}}

\newcommand{\pdot}{{\boldsymbol{\cdot}}}

\renewcommand{\=}{\stackrel\Sigma =}

\DeclareMathOperator{\tr}{tr}

\newcommand{\sss}{\scriptscriptstyle}

\renewcommand\geq{\geqslant}
\renewcommand\leq{\leqslant}

\DeclareMathOperator{\EXT}{d}
\newcommand{\ext}{{\EXT\hspace{.01mm}}}

\newcommand{\FFo}[1]{\mathring{\underline{\overline{\rm{#1}}}}}

\stackMath
\newcommand\reallywidehat[1]{%
\savestack{\tmpbox}{\stretchto{%
  \scaleto{%
    \scalerel*[\widthof{\ensuremath{#1}}]{\kern-.6pt\bigwedge\kern-.6pt}%
    {\rule[-\textheight/2]{1ex}{\textheight}}
  }{\textheight}%
}{0.5ex}}%
\stackon[1pt]{#1}{\tmpbox}%
}

\newcommand{\cg}{{\bm g}}

\begin{document}


\renewcommand{\today}{}
\title{
{The Dirichlet-to-Neumann Map
for \\[1mm] Poincar\'e--Einstein Fillings
}}

\author{ Samuel Blitz${}^\diamondsuit$, A. Rod Gover${}^\heartsuit$, Jaros\l aw Kopi\'nski${}^{\clubsuit}$, \&  Andrew Waldron${}^\spadesuit$}

\address{${}^\diamondsuit$
 Department of Mathematics and Statistics \\
 Masaryk University\\
 Building 08, Kotl\'a\v{r}sk\'a 2 \\
 Brno, CZ 61137} 
   \email{blitz@math.muni.cz}
 
\address{${}^\heartsuit$
  Department of Mathematics\\
  The University of Auckland\\
  Private Bag 92019\\
  Auckland 1142\\
  New Zealand,  and\\
  Mathematical Sciences Institute, Australian National University, ACT 
  0200, Australia} \email{r.gover@auckland.ac.nz}
  
  \address{${}^{\clubsuit}$
  Center for Theoretical Physics, Polish Academy of Sciences, Warsaw, Poland
} \email{jkopinski@cft.edu.pl}

  \address{${}^{\spadesuit}$
  Center for Quantum Mathematics and Physics (QMAP)\\
  Department of Mathematics\\ 
  University of California\\
  Davis, CA95616, USA} \email{wally@math.ucdavis.edu}

\vspace{10pt}

\renewcommand{\arraystretch}{1}

\begin{abstract}
We study the non-linear Dirichlet-to-Neumann map for the 
Poin\-car\'e--Einstein filling problem. For even dimensional manifolds the range of this non-local map is described in terms of  a rank two ``Dirichlet-to-Neumann tensor'' along the boundary determined by the Poincar\'e--Einstein metric. 
This tensor is 
 proportional to  the variation of renormalized volume along a path of Poincar\'e--Einstein metrics. 
We construct natural
 ``Dirichlet-to-Neumann hypersurface invariants'' that are conformally invariant and recover  
all Dirichlet-to-Neumann tensors.
We give an explicit formula for these  hypersurface invariants   and
 use a new vanishing result for odd order $T$-curvatures to show that they  are  the unique, natural conformal hypersurface invariant of 
transverse order equaling the boundary dimension.
We also construct
such conformally invariant Dirichlet-to-Neumann hypersurface invariants 
for Poincar\'e--Einstein fillings for odd dimensional manifolds 
 with conformally flat boundary.


\vspace{.8cm}

\noindent
\begin{center}
\end{center}

\vspace{-0.1cm}

\end{abstract}


\maketitle

\pagestyle{myheadings} \markboth{Blitz, Gover, Kopi\'nski, \& Waldron}{Poincar\'e--Einstein Dirichlet-to-Neumann Maps}



\newcommand{\balpha}{{\bm \alpha}}
\newcommand{\balphas}{{\scalebox{.76}{${\bm \alpha}$}}}
\newcommand{\bnu}{{\bm \nu}}
\newcommand{\blambda}{{\bm \lambda}}
\newcommand{\bnus}{{\scalebox{.76}{${\bm \nu}$}}}
\newcommand{\bnuss}{\hh\hh\!{\scalebox{.56}{${\bm \nu}$}}}

\newcommand{\bmu}{{\bm \mu}}
\newcommand{\bmus}{{\scalebox{.76}{${\bm \mu}$}}}
\newcommand{\bmuss}{\hh\hh\!{\scalebox{.56}{${\bm \mu}$}}}

\newcommand{\btau}{{\bm \tau}}
\newcommand{\btaus}{{\scalebox{.76}{${\bm \tau}$}}}
\newcommand{\btauss}{\hh\hh\!{\scalebox{.56}{${\bm \tau}$}}}

\newcommand{\bsigma}{{\bm \sigma}}
\newcommand{\bsigmas}{{{\scalebox{.8}{${\bm \sigma}$}}}}
\newcommand{\bbeta}{{\bm \beta}}
\newcommand{\bbetas}{{\scalebox{.65}{${\bm \beta}$}}}

\renewcommand{\bS}{{\bm {\mathcal S}}}
\newcommand{\bB}{{\bm {\mathcal B}}}
\renewcommand{\bT}{{\bm {\mathcal T}}}
\newcommand{\bM}{{\bm {\mathcal M}}}

\newcommand{\go}{{\mathring{g}}}
\newcommand{\nuo}{{\mathring{\nu}}}
\newcommand{\alphao}{{\mathring{\alpha}}}

\newcommand{\Ell}{\mathscr{L}}
\newcommand{\density}[1]{[g\, ;\, #1]}

\renewcommand{\Dot}{{\scalebox{2}{$\cdot$}}}

\newcommand{\PanE}{P_{4}^{\sss\Sigma\hookrightarrow M}}
\newcommand\eqSig{ \mathrel{\overset{\makebox[0pt]{\mbox{\normalfont\tiny\sffamily~$\Sigma$}}}{=}} }
\renewcommand\eqSig{\mathrel{\stackrel{\Sigma\hh}{=}} }
\newcommand\eqtau{\mathrel{\overset{\makebox[0pt]{\mbox{\normalfont\tiny\sffamily~$\tau$}}}{=}}}
\newcommand{\hd }{\hat{D}}
\newcommand{\hdb}{\hat{\bar{D}}}
\newcommand{\Two}{{{{\bf\rm I\hspace{-.2mm} I}}{\hspace{.2mm}}}{}}
\newcommand{\TwoN}{{\mathring{{\bf\rm I\hspace{-.2mm} I}}{\hspace{.2mm}}}{}}
\newcommand{\Fn}{\mathring{\mathcal{F}}}
\newcommand{\csdot}{\hspace{-0.75mm} \cdot \hspace{-0.75mm}}
\newcommand{\IdD}{(I \csdot \hd)}
\newcommand{\Kd}{\dot{K}}
\newcommand{\Kdd}{\ddot{K}}
\newcommand{\Kddd}{\dddot{K}}

 \newcommand{\bdot }{\mathop{\lower0.33ex\hbox{\LARGE$\cdot$}}}

\definecolor{ao}{rgb}{0.0,0.0,1.0}
\definecolor{forest}{rgb}{0.0,0.3,0.0}
\definecolor{red}{rgb}{0.8, 0.0, 0.0}

\newcommand{\APE}[1]{{\rm APE}_{#1}}
\newcommand{\PE}{{\rm PE}}
\newcommand{\FF}[1]{\mathring{\underline{\overline{\rm{\scalebox{.8}{$#1$}}}}}}
\newcommand{\ltots}[1]{{\rm ltots}_{#1}}

\newcommand{\FFdn}{\FF{\hh \hh d\hh\hh }^{\sss\rm DN}}

\newcommand{\FFdnf}{\FF{\, d\, }^{\sss\rm \hh DN^\flat}}

\section{Introduction}

A smooth Riemannian $d$-manifold $(M_+,g^o)$
is said to be conformally compact if
$M_+$ is the interior of a smooth manifold with boundary $M$, and  the {\it compactified metric}
$$
g_s=s^2 g^o
$$
extends smoothly to a metric on $M=\overline {M_+}$
for some (and so any)
 smooth defining function $s\in C^\infty  {M}$.
 By the latter we mean that $\ext s$ is nowhere vanishing along
 $\Sigma:=s^{-1}(0)=\partial M$. 
Conformally compact manifolds for which the trace-free Ricci tensor vanishes, {\it i.e.}
$$
Ric^{\hh g^o}_{(ab)\circ}=0\, ,
$$
are termed {\it Poincar\'e--Einstein}. 
Necessarily, Poincar\'e--Einstein manifolds have constant negative scalar curvature, which by convention, we choose to  be $-d(d-1)$.
 A classical example is the Poincar\'e ball 
$
M_+=\big\{\vec x\in {\mathbb R}^d|\, |\vec x|^{\hh 2}<1\big\}$
with
$$ g^o=\frac{4\hh |d\vec x|^{\hh 2}}{\, (1-|\vec x|^{\hh 2})^2}\, .
$$
In addition to myriad mathematical applications, Poincar\'e--Einstein  and conformally compact structures have played a central {\it r\^ole} in the geometry of physical models relating boundary field theories to bulk gravitational ones 
(see for example~\cite{FG,GJMS,GZscatt,CQY,
FGbook}
and~\cite{Maldacena,HS,GrWi}, respectively).

Since $s^2 g^o$  extends to the boundary for any defining function~$s$, a conformally compact structure canonically determines a conformal class of boundary metrics $\cc_\Sigma$ that are induced by a corresponding conformal class of metrics 
$\cc$ on $M$.
It turns out that  the formal asymptotics of a Poincar\'e--Einstein metric $g^o$ provide an effective tool for the study of the conformal manifold~$(\Sigma,\cc_\Sigma)$ (see~\cite{FG,FGbook}). We say that a conformally compact structure for which
$$
Ric^{g^o}_{(ab)\circ}=s^k Q_{(ab)\circ}\, ,
$$
for some smooth tensor $Q_{ab }$ on $M^d$, is {\it asymptotically Poincar\'e--Einstein of order $k$} and denote such data by $\APE{k}^d$.  
We often drop the label $d$ when context makes it clear.
 We also adopt the notation ${\mathcal O}(s^k)$ for the right hand side of the above display and other such quantities.
The case $k=d-2$ is distinguished in the sense that~$g_s$, modulo terms of order~${\mathcal O}(s^{d})$,
 is no longer determined 
solely by the boundary conformal manifold $(\Sigma,\cc_\Sigma)$.
This is exactly 
when
the asymptotics of a global solution first
probe the interior structure of $M_+$.

\smallskip

 A key global conformal invariant of an even dimensional Poincar\'e--Einstein structure $(M_+,g^o)$
is  its renormalized volume~\cite{HS,GrVol,RenVol}, which arises as follows. 
Firstly, a key result of Graham and Lee~\cite{GrahamLee} (valid for both even or odd dimension parity) is that, for each choice of boundary metric representative $\bar g$ for $\cc_\Sigma$, on a collar neighborhood of $\Sigma$ there exists a unique defining function $s_{\rm GL}$ such that 
this function returns the (minimal) geodesic distance $r$ to the boundary~$\Sigma$ as measured by the compactified metric $g_{r}=r^2 g^o$.
Now consider the one-parameter $\varepsilon\in {\mathbb R}_{> 0}$ family of Riemannian manifolds $M_\varepsilon$ determined by the metric $\bar g$ according to
$$
M_{\varepsilon}:=\{p\in M_+|r(p)>\varepsilon\}\, ,
$$
with Riemannian metric given by the restriction of the metric $g^o$. Then the {\it renormalized volume} is given, independently from the original choice of boundary metric $\bar g$ required to determine~$s_{\rm GL}$ and when $d$ is  even, by
$$
\operatorname{Vol}^{\rm ren}(M_+,g^o):=
\frac1{(d-1)!}\hh \frac{\ext^{d-1}}{\ext\varepsilon^{d-1}}\Big(\varepsilon^{d-1}
\int_{M_\varepsilon} \ext\! \operatorname{Vol}^{g^o}\Big)\Big|_{\varepsilon=0}\, .
$$
For four-manifolds, a  result of Anderson~\cite{Anderson} is that $$
\operatorname{Vol}^{\rm ren}= \frac{4\pi ^2}3 \chi - \frac1{24}\int_{M_+} |W^{g^o}|^2\,  \ext\! \operatorname{Vol}^{g^o}\, ,
$$
where $W$ denotes the Weyl tensor and $\chi$ is the Euler characteristic of $M_+$. Moreover, Anderson also established that the functional gradient of the renormalized volume, along a path of Poincar\'e--Einstein metrics, is given by
$$
\frac1{6} \otop\hh C^g_{\hat n (ab)}\, .
$$
Here $C$ denotes the Cotton tensor, $\hat n$ is the inward unit normal of $\Sigma \hookrightarrow (M,g)$, and $\otop$ denotes evaluation along $\Sigma$ as well as
projection of   $\odot^2 T^*M|_\Sigma$ to
directions tangential to $\Sigma$ and removal of the   trace part.
The resulting section space is isomorphic to that of
 $\odot^2_\circ T^*\Sigma$.  
Note that, since conformal transformations preserve angles (and metric tracelessness), this projection is independent of any choice of $g\in\cc$.
Throughout, we use a $\circ$  to denote 
trace-free parts of a tensor.

 As discussed above, picking a metric~$\bar g \in \cc_\Sigma$ and any $p$ in a suitable collar neighborhood of~$\Sigma$, we may write  the corresponding {\it Graham--Lee defining function} $
s_{\rm GL}(p)=r(p)
$
where  $r(p)$ is the minimal geodesic distance between~$p$ and $\Sigma$ 
as measured by the metric $g_r=r^2 g^o$. 
In coordinates~$(r,\vec x)$   for which~$g_r(\frac{\partial}{\partial r},\frac{\partial}{\partial \vec x})=0$, series in powers of~$r$ are termed {\it Fefferman--Graham expansions} following their fruitful employment for the description of conformal invariants in~\cite{FG}. In $d$ dimensions, terms of order $r^{d-1}$ in the expansion of the {\it Graham--Lee compactified metric}~$g_r$ are of particular interest.
In particular, in dimension $d=4$, the third  Lie derivative of the {compactified metric}~$g_r$
extracts the term of order $r^3$ and 
defines a tensor 
$$
 \operatorname{DN}_{ab}^{(4)}:=
\big({\mathcal L}^3_{\frac{\partial}{\partial r}} g_r\big)_{ab}\big|_\Sigma
\, .
$$
Remarkably, the above tensor satisfies
$$
 \operatorname{DN}_{ab}^{(4)}\propto \otop C^{g_r}_{\hat n (ab)}\, ,
 $$
 where $\propto$ denotes equality up to a non-zero constant.
This tensor is interesting,
 especially  for general relativity (see for example~\cite{GoKo,Herfray}), and because it is the first odd order term in the Fefferman--Graham  expansion of~$g_r$. 
As shown in~\cite{FG,GrVol},   the tensor $\operatorname{DN}^{(4)}$ changes covariantly  when computed with respect to an expansion in a new geodesic coordinate~$r^\prime$ corresponding to a different choice of $\bar g^\prime \in \cc_\Sigma$.
On the other hand, under a general conformal transformation of the metric
\begin{equation}\label{Ccov}
 \otop\hh C^{\Omega^2 g}_{\hat n^\prime (ab)}
=\bar \Omega^{-1} \otop\hh  C^{g}_{\hat n(ab)} \, ,
\end{equation}
where $\hat n^\prime_a =\bar  \Omega\hh  \hat n_a$, for $0<\Omega \in C^\infty M$ and $\bar \Omega := \Omega|_\Sigma$, so long as $g\in \cc $ where $\cc$ is a conformal class of metrics defined by a Poincar\'e--Einstein structure (see~\cite{BGW} for the generalization to arbitrary conformal structures).
It therefore constitutes an example of what we shall later term a conformal hypersurface invariant of the conformal embedding~$\Sigma\hookrightarrow (M,\cc)$, while
 the tensor~$\operatorname{DN}^{(4)}$
 defines a section of~$\odot_\circ^2 T^*\Sigma[-1]$. (This is the space of symmetric, rank 2, trace-free tensor-valued,
 weight~$-1$ conformal densities; the latter are explained in Section~\ref{conformalstuff} but here it suffices to view a weight~$w$ tensor-valued density $T$ on a conformal manifold $M$ as an equivalence class $T=[g,t]=[\Omega^2 g, \Omega^w t]$ of metric-tensor pairs with respect to conformal rescalings.)
 In the same vein, the tensor $\otop C^{g}_{\hat n(ab)}$
 defines a section of 
  $\otop\!\odot^2\! T^*M[-1]|_\Sigma$ that  encodes the information of $ \operatorname{DN}^{(4)}$.

 Since the section of~$\odot_\circ^2 T^*\Sigma[-1]$ defined by the tensor $ \operatorname{DN}^{(4)}$ is determined
 in terms of $g^o$, 
it is an invariant of the Poincar\'e--Einstein structure.
However it is not determined by the local data of the boundary conformal class $\cc_\Sigma$, again see~\cite{FG}. 
Knowledge of $\operatorname{DN}^{(4)}$  requires information about the global Poincar\'e--Einstein metric~$g^o$. 
Our point is that, in four dimensions,  the tensor $\otop C^{g}_{\hat n(ab)}$ 
captures  the range
 of the 
global, non-linear, Dirichlet-to-Neumann map for Poincar\'e--Einstein metrics introduced by Graham in~\cite{LittleRobin}.


A natural tensor, such as $\otop C^{g_r}_{\hat n(ab)}$, 
which depends on~$g_r$ through   three (and possibly fewer) derivatives  with respect to a coordinate~$r$ transverse to~$\Sigma$, 
is an example of what will
 be
 termed a 
natural hypersurface
tensor of  transverse order $3$ (see Section~\ref{D2N}). The natural hypersurface   tensor $\otop C^{g_r}_{\hat n(ab)}$ is special, in that it obeys the conformal covariance property of Equation~\eqref{Ccov}.
Broadly, our aim  is to construct conformally covariant, natural hypersurface tensors of 
transverse order $d-1$ in even dimensions~$d$, that capture the Dirichlet-to-Neumann data $\operatorname{DN}^{(d)}$.

%
%
%

In arbitrary even dimensions $d$, the coefficient $({\mathcal L}_{\frac{\partial}{\partial r}})^{d-1} g_r\big|_\Sigma$ in the Fefferman--Graham expansion of a Poincar\'e--Einstein metric also
changes covariantly  when computed with respect to a different choice of 
boundary metric representative~$\bar g\in \cc_\Sigma$,
but once again is not determined by the boundary data~$(\Sigma,\cc_\Sigma)$~\cite{GrVol}. It therefore defines 
an invariant of the 
Poincar\'e--Einstein structure which
 we term a  {\it (Poincar\'e--Einstein) Dirichlet-to-Neumann tensor}~$\operatorname{DN}^{(d)}$. When it exists (meaning when the data of $(\Sigma, \cc_\Sigma)\hookrightarrow M^d$ uniquely determines a corresponding Poincar\'e--Einstein structure, for example for boundary metrics suitably close to the round sphere~\cite{GrahamLee}), the map $$
  \cc_\Sigma
 \mapsto \operatorname{DN}^{(d)}\in \Gamma(\odot_\circ^2 T^*\Sigma[3-d])$$ is termed the {\it (Poincar\'e--Einstein) Dirichlet-to-Neumann map}; see~\cite{LittleRobin}. The linearization of this map is studied in~\cite{Wang}.
The data of the boundary conformal class of metrics $\cc_\Sigma$ and the tensor~$\operatorname{DN}^{(d)}$  determines the formal Fefferman--Graham asymptotics of $g_r$  to all orders~\cite{FG}.

\smallskip

 A first result here is the existence of a  unique natural tranverse order 5 tensor for 
  dimension~$6$ Poincar\'e--Einstein structures that captures the Dirichlet-to-Neumann tensor.
\begin{theorem}\label{blunt}
Let $(M_+^6,g^o)$ be a Poincar\'e--Einstein structure with conformal infinity $\Sigma$.
Then the unique natural transverse order 5 section of $\otop\!\!\odot^2\! \!T^*M[-3]|_\Sigma$ 
that is an invariant of the 
Poincar\'e--Einstein structure is given, up to a non-zero constant multiple, for a choice of $g$ in the conformal class $\cc$ determined by $g^o$, by
$$
\VIo_{ab}^{{\sss\rm DN}}:=
\otop \big((\nabla_{\hat n}+2H)B_{ab}\big)
-4 
\bar C_{c(ab)}
\bar \nabla^c H\, .
$$ 
Furthermore
$$
\VIo^{\sss\rm DN}\simeq \operatorname{DN}^{(6)} .
$$
 \end{theorem}
 \noindent
  The detailed notions of transverse order and  natural tensors are described in Section~\ref{D2N}.  
 In the above, $B_{ab}$ denotes the  Bach tensor, $H$ is the mean curvature of the embedding~$\Sigma\hookrightarrow (M,g)$, and bars are used to denote objects intrinsic to the  hypersurface $\Sigma$. 
 Also, since $\VIo^{\sss\rm DN}$ and $\operatorname{DN}^{(6)}$ live in differing section spaces, namely that of $\otop\odot^2T^*M[-3]|_\Sigma$ and $\otop\odot^2T^*\Sigma[-3]$ respectively, we have employed a notation $\simeq$. 
  For $A\in  \otop\odot^2\!T^*M[3-d]|_\Sigma$ and $B\in \otop\odot^2\! T^*\Sigma[3-d]$, we say 
 $$A\simeq B$$ if when given any $\bar g \in \cc_\Sigma$ and its corresponding Graham--Lee compactified metric $g_r$, then their respective evaluations obey
 $$
 A^{g_{r}} \propto B^{\bar g}\, .
 $$
Note that,  as bundles,  $\otop\odot^2T^*M[-3]|_\Sigma$ and $\otop\odot^2T^*\Sigma[-3]$  are in fact isomorphic. However, on the one hand, the tensor $\operatorname{DN}^{(6)}$ is defined invariantly with respect to choices of $\bar g \in \cc_\Sigma$, while~$\VIo^{\sss\rm DN}$ is invariant with respect to choices of~$g\in  \cc$. Therefore we will use both notations, rather than invoking this isomorphism. Also strictly, since $\otop$ includes restriction to $\Sigma$ in its definition, we could drop the moniker $|_\Sigma$, but shall not do so for reasons of emphasis. 
Theorem~\ref{blunt} 
can be  proved by exhaustion;  key details are provided in Appendix~\ref{wedetestappendices}.

\smallskip

A main point of this article is
to both give a simple characterization of these maps and, in particular,  compute conformal hypersurface invariants, denoted~$\FF{\hh\hh d\hh\hh }^{\sss\rm DN}$ and termed {\it Dirichlet-to-Neumann hypersurface invariants}, that each capture the  Dirichlet-to-Neumann tensor $\operatorname{DN}^{(d)}$
for dimensions $d\geq 4$.
 Critically, Dirichlet-to-Neumann tensors are natural 
 and of  transverse order $d-1$; see  Proposition~\ref{Redford}. 
{\it I.e.}, we shall establish higher dimensional  analogs of Theorem~\ref{blunt}. 

These results
dovetail nicely
 with a seminal physics conjecture of Deser and Schwimmer~\cite{DS} that was proved by Alexakis~\cite{AlexakisI,AlexakisII},
as well as work on renormalized volumes by Anderson~\cite{Anderson} and Chang, Qing and Yang~\cite{CQY}. The latter work builds on  the former ones  to show that general even dimensional renormalized volumes are a sum of
an  Euler characteristic term plus an integral over~$M_+$ whose integrand  is a natural local conformal invariant.
In principle, following Anderson, even dimensional Dirichlet-to-Neumann tensors
could be computed by varying renormalized volumes~\cite{Anderson}.
The latter are intimately related to~$Q$-curvature invariants, whose
  complexity in terms of Riemannian invariants explodes in 
eight  and higher dimensions~\cite{GOpet}, so a
more powerful method is required.
 Luckily,  tractor calculus methods, and in particular recent results on conformally invariant normal operators~\cite{GPt} and conformal fundamental forms~\cite{BGW}, yield a rather  simple solution to this problem. This is a main theorem:



\begin{theorem} \label{peachy} 
Let $(M_+^{d},g^o)$ be an even-dimensional Poincar\'e--Einstein structure with conformal infinity $\Sigma$ and $d\geq 6$.
For $d=6$, we have that 
$$\VIo^{\sss\rm DN} = \bar{q}^* \otop \delta_{R}\, W\, .$$
Also, for $d\geq 8$,  set
$$
\FF{\hh \hh d\hh\hh }^{\sss\rm DN}
:=
\bar q^* \otop
 \delta_{\frac{d-6}{2},\frac{d-4}{2}}\, W
\, .
$$ 
Then, for $d\geq 6$, 
$$
\FF{\hh \hh d\hh\hh }^{\sss\rm DN}
\simeq
\operatorname{DN}^{(d)}\, .
$$
Moreover,  $\FF{\hh \hh d\hh\hh }^{\sss\rm DN}$ 
 has 
leading transverse order term 
$$
\otop\hh  \colon \nabla_{\hat n}^{d-5}
\colon\hh
B_{ab}\, ,
$$ 
and  is the 
unique (up to a non-zero constant multiple) natural conformal hypersurface invariant of  transverse order~$(d-1)$ in $\Gamma(\otop \! \odot^2 \! T^* M[3-d]|_\Sigma)$ 
determined by the 
Poincar\'e--Einstein structure.
Thus it is the functional gradient of the renormalized volume
along a path of Poincar\'e--Einstein metrics.
\end{theorem}
\noindent
Here we have employed a normal ordering notation
$\colon\nabla_{\hat n}^{\ell}\colon$ for $
\hat n^{a_1}\cdots \hat n^{a_\ell}
\nabla_{a_1}\cdots \nabla_{a_\ell}$
(the composition of operators $\nabla_{\hat n} \circ \nabla_{\hat n}$ is ill-defined) while 
the above tractor calculus notions and notations, including the $W$-tractor~$W$,  are explained in Section~\ref{conformalstuff}. 
The normal operator $ \delta_{\frac{d-6}{2},\frac{d-4}{2}}$ acting on the $W$-tractor is an example of those constructed in~\cite{GPt} and generalized to act on tractors in~\cite{B}, see Equation~\eqref{normalops}.

 It is interesting to note that in the context of minimal submanifolds of Poincar\'e--Einstein manifolds, there is a natural analog of the Dirchlet-to-Neumann map studied here~\cite{AM,MK}, that is again linked to variations of renormalized volume. We expect there to be a version of our construction applicable to that setting.

\medskip


The recently developed conformal fundamental forms of~\cite{BGW,B}, designed for the study of  conformal hypersurface  embeddings $\Sigma\hookrightarrow (M,\cc)$, play a key {\it r\^ole} in establishing our main theorems for the case~$d$ is even.
Given such an embedding, 
a $k^{\rm th}$ {\it conformal fundamental form} 
$\FF{k}$ is any section of~$\otop\!\odot^2 T^*M[3-k]|_\Sigma$ that is a natural conformal hypersurface invariant of transverse order~$k-1$ (in many contexts the integer $k$ is denoted by Roman numerals). In the cases $2\leq k\leq d-1$,~$k^{\rm th}$~conformal fundamental forms are  obstructions to a $d$-dimensional conformally compact manifold  being an $\APE{d-3}$ structure. 
This generalizes an older result~\cite{LeBrun,Goal}
 that the hypersurface embedding $\Sigma\hookrightarrow(M,s^2 g^o)$  (for any defining function $s$) of a Poincar\'e--Einstein structure $(M_+,g^o)$ is necessarily umbilic---so the trace-free second fundamental form $\IIo$ vanishes. In fact, vanishing of all 
conformal fundamental forms for 
$2\leq k\leq \ell$ with $\ell\leq d-1$
is in fact a both necessary and sufficient condition for a conformally compact structure to be 
$\APE{\ell-2}$.
Moreover, we have denoted Dirichlet-to-Neumann hypersurface invariants by $\FF{\hh\hh d\hh \hh}^{\sss\rm DN}$
because they may be viewed as critical order conformal fundamental forms, whose job is no longer to measure obstructions to the Poincar\'e--Einstein condition but rather to extract Dirichlet-to-Neumann data.

In dimension four a conformally invariant fourth fundamental form is~\cite{BGW,GoKo} 
\begin{equation}\label{mynumberisugly}
\IVo_{ab} :=
C_{\hat  n(ab)}^\top  
-  \bar \nabla^c W_{c(ab) \hat n}^\top+
H W_{\hat n ab \hat n} \, .
\end{equation}
For conformal embeddings arising from an $\APE{1}$ structure, the conformal hypersurface invariant~$\IVo$ restricts to
$\IVo^{\sss\rm DN}\simeq\operatorname{DN}^{(4)}$.
Importantly, unlike the tensor $\operatorname{DN}^{(4)}$, the fourth fundamental form is a natural conformal hypersurface invariant for {\it any} conformal hypersurface embedding, rather than only those embeddings arising as the boundary of a Poincar\'e--Einstein structure.
An interesting question is whether there are~$d^{\rm th}$ conformal fundamental forms for more general conformal embeddings. 

\medskip 

The case of odd dimensional Poincar\'e-Einstein structures is subtle. Firstly, no Anderson-type formula for the renormalized volume is available, and in general this quantity  depends on a choice of boundary metric representative $\bar g\in \cc_\Sigma$. Concordantly, there is an obstruction to smoothness of the Fefferman--Graham expansion termed the obstruction tensor~\cite{FG}. 
For general boundary conformal structures, the $r^{d-1}$ coefficient 
 in the Fefferman--Graham expansion of a Poincar\'e--Einstein metric does {\it not} define a tensor-density 
valued in~$\odot_\circ^2 T^*\Sigma[3-d]$
 independently of the choice of~$\bar g\in \cc_\Sigma$. 
 Indeed 
 the quantity $({\mathcal L}_{\frac{\partial}{\partial r}})^{d-1} g_r$ may not extend smoothly to the boundary $\Sigma$. 
  When the Fefferman--Graham obstruction tensor vanishes we may then at least define $\operatorname{DN}^{(d)}:=({\mathcal L}_{\frac{\partial}{\partial r}})^{d-1} g_r\big|_\Sigma$ for $d$ odd. However,  this  does not in general transform conformally covariantly as one changes the  choice of boundary metric~$\bar g$ required to determine the Graham--Lee compactified metric. But, when 
the boundary~$(\Sigma,\cc_\Sigma)$ is  conformally flat, we can instead define a canonical section~$\FF{\, d\, }^{\sss\rm \hh DN^\flat}$
of $\otop\!\!\odot^2 \!T^*M[3-d]|_\Sigma$
which is an invariant of the Poincar\'e--Einstein structure.
 \begin{theorem}\label{spliff}
Let $(M_+^5,g^o)$ be a Poincar\'e--Einstein structure with (locally) conformally flat conformal infinity $\Sigma$.
Then the unique (up to a non-zero constant multiple) natural transverse order 4 section of $\otop \odot_\circ^2 T^*M[-2]|_\Sigma$ that is an invariant of the Poincar\'e--Einstein structure, is given, for a choice of $g\in \cc$, by
$$
\Vo^{\sss\rm DN^\flat}:=
 B^\top_{(ab)\circ}\, .$$ 
Moreover, evaluated on  the Graham--Lee compactified metric $g_r$ corresponding to 
 a (locally) flat boundary metric,
$$
\Vo^{\sss\rm DN^\flat}\simeq \operatorname{DN}^{(5)}\, .
$$
\end{theorem}

 \smallskip
 
 \noindent
 The above is an odd dimensional analog of Theorem~\ref{blunt} and the following is that of Theorem~\ref{peachy}.
  
 \begin{theorem}\label{theoremono} 
Let $(M^d_+,g^o)$ be an odd-dimensional Poincar\'e--Einstein structure with a (locally) conformally flat conformal infinity and $d\geq 5$. Then if $d=5$, 
we have that 
$$
\Vo^{\sss\rm DN^\flat} =
 \otop\circ
 q^*\big(\big[\hat n^a \colon \nabla_{\hat n} \colon  F_{ab}{}^A{}_B\big]\big|_\Sigma\big)\, .
$$
For $d\geq 7$, set
$$
 \FF{\, d\, }^{\sss\rm \hh DN^\flat}:=
 \otop\circ
 q^*\big(\big[\hat n^a \colon \nabla_{\hat n}^{d-4} \colon  F_{ab}{}^A{}_B\big]\big|_\Sigma\big)\, .
$$
Then, evaluated on  the Graham--Lee compactified metric $g_r$ corresponding to 
 a 
 (locally) flat boundary metric,
$$
 \FF{\, d\, }^{\sss\rm \hh DN^\flat}\simeq
 \operatorname{DN}^{(d)}\, .
$$
Moreover,
$\operatorname{DN}^{(d)} := ({\mathcal L}_{\frac{\partial}{\partial r}})^{d-1} g_r\big|_\Sigma$ defines an element of $\Gamma(\odot_\circ^2 T^*\Sigma[3-d])$.

\end{theorem}

 The article is structured as follows.
 Section~\ref{conformalstuff}
 introduces the main technologies we require for handling 
 conformal geometries and
hypersurfaces embedded therein (tractor and hypersurface {\it cognoscenti} might safely leapfrog this section).
Our main results are proved in Section~\ref{D2N}.

\subsection{Conventions}
Throughout we take $M$ to be a smooth, for simplicity oriented, $d\geq 3$-dimensional manifold, and throughout unless otherwise specified, all structures are taken to be smooth. The canonical $d$-form determined by a metric $g$ is denoted $\ext {\rm Vol}^g$ and may be used when integrating over $M$; this is normalized  such that in local coordinates it gives the measure $\sqrt{\det g}\, \ext x^1\cdots \ext x^d$ where $g$ here denotes the matrix of metric components. The tangent, cotangent, and tensor bundles of~$M$ are denoted, respectively,  by $TM$, $T^*M$, and ${\mathbb T}M$. Sections  are often handled using an abstract index notation; for example $t^{ab}{}_c\in \Gamma(\otimes^2 TM\otimes T^*M)$, and the contraction $v(\omega)$ of a vector field $v$ and one-form  field $\omega$ is denoted by $v^a \omega_a=v_\omega=\omega_v\in C^\infty M$.  The inverse of a metric $g_{ab}\in \Gamma(\odot^2 T^*M)$  is denoted $g^{ab}$ and can be used to ``raise indices'' in the standard way, for example $\omega^a=g^{ab} \omega_b$. The symmetric trace-free part of a tensor $X_{ab}$ is denoted $X_{(ab)\circ}:=\frac12(X_{ab}+X_{ba}) - \frac1d g_{ab} X_c{}^c\in \Gamma(\odot_\circ^2 T^*M)$. 
Also, for a tensor $X_{abc\cdots}$, we denote $|X|_g^2:= X_{abc\cdots} X^{abc\cdots}$.
While we work solely in Riemannian signature, many results carry over {\it mutatis mutandis} to indefinite
metric signatures.

 The Levi-Civita connection of a metric is denoted $\nabla^g$ (or simply $\nabla$ when context lends clarity; we will similarly drop the superscript $g$ on other Riemannian tensors).
The Riemann tensor $R$ of $\nabla^g$ is defined by
$$
\big(R(x,y)z\big)^a:=\big((\nabla_x \nabla_y - \nabla_y  \nabla_x -\nabla_{[x,y]})z\big)^a = x^c y^d R_{cd}{}^a{}_b z^b = R_{xy}{}^a{}_z\in \Gamma(TM)\, ,
$$
where  $x,y,z\in \Gamma(TM)$ and $[x,y]$ is their Lie bracket. The Weyl, Cotton and Bach tensors are given by 
\begin{align}
W_{abcd}&:= R_{abcd}-g_{ac} P_{bd}
+g_{ad} P_{bc}
+g_{bc} P_{ad}
-g_{bd} P_{ac}\, ,\nonumber\\
C_{abc}\:&:=\nabla_a P_{bc}-\nabla_b P_{ac}\, ,\nonumber\\
B_{ab}\:\:&:=\Delta P_{ab}-\nabla^c \nabla_a P_{bc}
+P^{cd}W_{acbd}\, ,\label{Johanne}
\end{align}
where the Schouten tensor $P_{ab}$ and its trace $J=P_a{}^a$ are defined by the equation
$Ric_{ab}:=R_{ca}{}^c{}_b=(d-2)P_{ab}+g_{ab}J$, and $\Delta:=\nabla^a \nabla_a$ is the (negative) rough Laplacian.
We will often place a symbol above an equals sign to qualify its domain of applicability, for example $A\stackrel\Sigma =B$ implies equality along the hypersurface $\Sigma$.
We also rely on the canonical isomorphism between the projection of the bulk tangent bundle along a hypersurface and the hypersurface tangent bundle to employ the same abstract indices for hypersurface tensors as for their bulk counterparts.

\section{Conformal Hypersurface Calculus}\label{conformalstuff}

A {\it conformal manifold} $(M,\cc)$ is  a smooth manifold equipped with a  conformal class of metrics~$\cc$, meaning that if $g,g'\in \cc$ then  $g'=\Omega^2 g$ for some $0<\Omega\in C^\infty M$. A {\it conformal hypersurface embedding}  $\Sigma\hookrightarrow (M,\cc)$ is a conformal manifold equipped with a smoothly embedded codimension~1 submanifold~$\Sigma$. The data $(M,\cc)$ may be viewed as a  ray subbundle of $\odot^2T^*M$ and in turn as an  ${\mathbb R}_+$-principal bundle. The group action  $t\mapsto 
s^{-w}t$ on $t\in {\mathbb R}$ (corresponding to $g(p)\mapsto s^2 g(p)$) for $w\in {\mathbb R}$ and $s\in {\mathbb R}_+$ determines an associated line bundle~$\ce M[w]$ over~$M$ called a {\it weight $w$ conformal density bundle}. 
Given any vector bundle~${\mathcal V}M$ over~$M$, we denote ${\mathcal V}M[w]:={\mathcal V}M\otimes \ce M[w]$. 
It is also useful to define the operator $\underline w$ which acts by multiplication by $w$ on tensor-valued densities $\Phi$ of weight $w$, so that $\underline w\Phi = w \Phi$. 

The tautological section ${\bm g}\in\Gamma(\odot^2T^*M[2])$ determined by~$\cc$ is called the {\it conformal metric}. We may equally well label a conformal manifold $(M,\cc)$ by the pair $(M,\cg)$.
The {\it tractor bundle} is defined by (see~\cite{Curry})
$$
\ct M:=\Big(\bigsqcup_{g\in \cc} \ct^g M\Big)\Big \slash \sim\, ,
$$
where 
\begin{equation}\label{DC}
\ct_g M:=\ce M[1]\oplus T^*M[1]\oplus \ce M[-1]\, .
\end{equation}
The  equivalence relation $\sim$ on direct sum bundles is defined by the following relation on sections
$$
\Gamma(\ct_{g'})\ni (\tau,\mu + \Upsilon \tau, \rho - \Upsilon.\mu -\frac12 \Upsilon^2 \tau)\sim  
(\tau,\mu,\rho)\in \Gamma(\ct_{g})\, ,
$$
for $\Upsilon := \ext \log \Omega$. Here $\Upsilon.\mu:={\bm g}^{-1} (\Upsilon,\mu)$ and $\Upsilon^2:={\bm g}^{-1} (\Upsilon,\Upsilon)$. 
There is a canonical section $X$ of $\ct M[1]$ given by
$$
(0,0,1)\in \ct_g M[1]\, ,
$$
for any $g\in \cc$, which is termed the {\it canonical tractor}. Also, there is an indefinite, conformally invariant,  bundle metric $h\in \Gamma(\odot^2 \ct^*M)$ given, in an obvious matrix notation, by
$$
\begin{pmatrix}
0&0&1\\
0&{\bm g}^{-1}&0\\
1&0&0
\end{pmatrix}\, ,
$$
and termed the {\it tractor metric}.
Sections $T\in\Gamma(\ct M)$ can be denoted in an abstract index notation by $T^A$, so the tractor metric is then $h_{AB}$ and will be used to raise and lower tractor indices in the standard way. For example, $h_{AB} X^A X^B$ is the zero section of $\ce M[2]$.

On an oriented conformal hypersurface $\Sigma \hookrightarrow (M,\cc)$, there is a canonical section $N$ of~$\ct M\big|_\Sigma$ given, for any $g\in \cc$, by
$$
(0,\hat n,-H)\in \ct_g M\, ,
$$
where $\hat n$ is the inward unit conormal of $g$ and $H$ its mean curvature ({\it i.e.}, the average of the eigenvalues of the second fundamental form). This section  is called the {\it normal tractor}~\cite{BEG} and is key  for the study of conformally embedded hypersurfaces. 
The normal tractor and tractor metric allow us to define the tractor analog of the projection operator $\otop$ in the obvious way;  we recycle the same notation for this. 
There is a particularly useful relationship between $N$ and defining functions for $\Sigma$, but we first must introduce further tractor technology.

The tractor bundle~$\ct M$ comes equipped with a  canonical connection $\nabla^\ct$, termed the {\it tractor connection}, defined for any $g\in \cc$, and $ T\stackrel{\sss{\mathcal T}_g}{:=}(\tau,\mu,\rho)\in \Gamma(\ct_g M)$ by 
\begin{equation}\label{trac-conn}
\nabla^\ct  T\stackrel {\sss{\mathcal T}_g}=(\nabla \tau
-\mu
,\nabla \mu  
+ P^{g} \tau
+\bm g \rho  ,\nabla\rho
-P^g.\mu) \, .
\end{equation}
We will  adorn equal signs with  the symbol $\sss{\mathcal T}_g$
to indicate that we have used the metric $g$ to decompose sections of (tensor products of) the  tractor bundle into a direct sum of density valued sections according to Equation~\nn{DC}.

The curvature of $\nabla^\ct$ is denoted by $F$, or $F_{ab}{}^A{}_B$ in an abstract index notation.
In the above $P^g.\mu:=P({\bm g}^{-1}(\mu,\pdot),\pdot)$. 
Also, in the above, $\nabla$ is the Levi--Civita connection $\nabla^g$ of~$g$.
If follows from our discussion of conformal densities that, given $v\in\Gamma(TM)$ and a weight $w$ density $\varphi=[g;f]$, one has $\nabla^g_v \varphi = [g,(\ext f)(v)]$. 
Note that a conformal density $0<\tau\in \Gamma(\ce M[1])$ defines a metric $g=\tau^{-2}\cg\in \cc$ and conversely a metric $g\in \cc$ determines a strictly positive  density $\tau_g:=\big( \ext\! \operatorname{Vol}^\cg\! / \ext\! \operatorname{Vol}^{g}\!\big)^{\frac1{d}}\in\Gamma(\ce M[1]) $. 
We will term such a scale {\it true}.
Thus, given $g\in \cc$, we can define a canonical section 
$$
Z^A_a:=\tau_g\hh \nabla^{\ct}_a \big(\tau_g^{-1}X^A \big)\in\Gamma(T^*M\otimes \ct M[1])\, .
$$

Tractor tensor bundles are defined in the standard way and sections of such can also be denoted by an abstract index notation, for example $T^{AB}{}_C\in \Gamma({\otimes^2}\ct M\otimes \ct^* M)$; we 
will often label a generic tractor tensor bundle by $\ct^\Phi M$ (in general  $\Phi$ is some representation of $SO(d+1,1)$).
Then given a tractor $T \in \Gamma(\ct^\Phi M[w])$,  the {\it Thomas-D operator} 
$$D:\Gamma(\ct^\Phi M[w])\to \Gamma(\ct M \otimes \ct^\Phi M[w-1])\, ,$$ 
is defined~\cite{BEG} acting on $T$,
 for any~$g\in \cc$, by
\begin{equation*}
\resizebox{.9\hsize}{!}{$
D \, T = \left( w(d+2w-2) T, (d+2w-2) \nabla^\ct T, -\Delta^\ct T - w\bg^{ab} P_{ab}^g T \right) \in \Gamma(\ct _g M \otimes \ct^\Phi M[w-1])\,,
$}
\end{equation*}
where $\Delta^\ct := {\bm g}^{ab} \nabla^\ct_a \nabla^\ct_b$ is the
tractor Laplacian 
 and $\nabla^\ct$ denotes the tractor-Levi--Civita coupled connection.
When $w\neq 1-\frac d2$, we define $\hd :=(d+2w-2)^{-1} D$.

There is a useful  tractor analog of the Weyl tensor:
the 
{\it $W$-tractor}
 is a canonical section of~$
\otimes^4 \ct^*M[-2]
$,
with Weyl-tensor symmetries,
defined
in dimensions~$d\geq5$
for any choice of~$g\in \cc$,
by 
$$
W^{ABCD}\stackrel {\sss{\mathcal T}_g}=Z^A_a Z^B_b Z^C_c Z^D_d W^{abcd}+
4 Z^{[A}_{a\phantom b} Z^{B]}_b Z^{[C}_{c\phantom b} X^{D]}_{\phantom b} C^{abc}+
 \tfrac{4}{d-4}Z^{[A}_{a\phantom b} X^{B]}_{\phantom b} Z^{[C}_{c\phantom b} X^{D]}_{\phantom b} B^{ac}\, ,
$$
where $W_{abcd}$, $C_{abc}$ and $B_{ab}$
are, respectively,  the Weyl, Cotton and Bach tensors of the metric~$g$. Note that
$$
F_{ab}{}^{CD}=
Z_a^A Z_b^B
W_{AB}{}^{CD}
=Z^C_c Z^D_dW_{ab}{}^{cd}+
2 Z^{[C}_{c\phantom b} X^{D]}_{\phantom b} C_{ab}{}^c\, 
;
$$
see~\cite{GOadv,GOpet,BGW} for further details. Note that some authors define a $W$-tractor in dimensions $d\geq 5$ by multiplying  the above definition by a factor $(d-4)$
so that it is defined in all dimensions $ d\geq 3$.

\smallskip

Indeed there is a general notion termed the {\it projecting part} of a tractor, of which we only need special cases. In particular for a tractor $T \in \Gamma(\ct^{\otimes 4} M[w])$,
given in a choice of $g\in \cc$ by 
$$
T^{ABCD}\stackrel {\sss{\mathcal T}_g}= 4Z^{[A}_{a\phantom b} X^{B]}_{\phantom b} Z^{[C}_{c\phantom b} X^{D]}_{\phantom b} t^{ac}\, ,
$$
such that $t^{ab}\neq 0$, 
 we can extract
 the projecting part $q^*(T^{ABCD}):=[g;t_{ab}]\in \Gamma(\otimes^2 T^* M[w])$. 
Similarly for a tractor $T^{AB}\stackrel {\sss{\mathcal T}_g} = 2 Z^{[A}_b X^{B]} t^b$, we have that $q^*(T^{AB})=[g;t^b]$.

\smallskip

Given a conformal hypersurface embedding $\Sigma \hookrightarrow (M,\cc)$, we call $\sigma\in \Gamma(\ce M[1])$ a {\it defining density}
if $\sigma=s\tau$ for some $0<\tau\in \Gamma(\ce M[1])$ and 
$s$  a defining function for $\Sigma$. 
Then we may define a metric
$$
g^o=\sigma^{-2} {\bm g}
$$
on $M_+=M\setminus \Sigma$. Because $g^o$ is not defined  along $\Sigma=\partial M_+$, we refer to it as a {\it singular metric}. An important theorem of~\cite{Goal} is that $(M_+,g^o)$ is Poincar\'e--Einstein precisely when the {\it scale tractor} $I_\sigma:=\hat D \sigma$ is parallel with respect to the tractor connection $\nabla^\ct$
 and normalized,~{\it i.e.},
$$
\nabla^\ct I_\sigma = 0= I_\sigma^2-1 \, ,
$$ 
where $I_\sigma^2:=h(I_\sigma,I_\sigma)$.
This allows us to  label a  Poincar\'e--Einstein structure by the data~$(M,\bg,\sigma)$,
 where 
 the singular Einstein metric $g^o=\sigma^{-2}\bg$.
Note that, as was discussed in the introduction, the $k$th order asymptotics of this data can  be labeled by a representative triple $(M^d ,g_k,s)$ for which $\cc_k=[g_k]$ and~$\sigma=[g_k;s
]$, and such that $(M_+,\sigma^{-2}g_k)$ is $\APE{k}^d$.
We recall  that for $\APE{k\geq 0}^d$ structures, the normal tractor to $\Sigma$ is  given by~\cite{Goal} 
\begin{equation}\label{normalT}
N=I_\sigma|_\Sigma\, .
\end{equation}

The final piece of conformal hypersurface technology that we need is  
a generalization of the scalar
normal operators of~\cite{GPt}. For the simplest of these, recall that,  the {\it conformal-Robin operator}
$
\delta_R : \Gamma(\ce M[w])\to \Gamma(\ce \Sigma[w-1])
$
is defined along~$\Sigma$, for any $g\in \cc$, and $[g;f]\in  \Gamma(\ce M[w])$ by~\cite{Cherrier} 
$$
\delta_R f = \big(\nabla_{\hat n} - wH^g\big) f\hh \big|_\Sigma\, .
$$
This generalizes directly to tractors by replacing the Levi--Civita connection in  the above by its tractor-coupled analog. 

\smallskip

If $s$ is any defining function for $\Sigma$, we say that an operator ${\sf L}:\Gamma(\ct^\Phi M[w])  \to\Gamma(\ct^{\Phi'}M[w'])|_\Sigma$
has {\it transverse order} $k$ if
\begin{equation}\label{TO}
{\sf L} (s^{k+1} U)=0\neq {\sf L} (s^{k} V)\, ,
\end{equation}
for every  smooth tractor $U$ and some smooth tractor $V$ in its domain. Clearly $\delta_R$ has transverse order~$1$. The {\it normal operators} $\delta_k$ defined by
\begin{equation}\label{roughlyhalf}\Gamma(\ct^\Phi M[w]) \ni T \stackrel{\delta_k}\longmapsto 
\hh \colon (N^A D_A)^k\colon\hh\hh
T
:=N^{A_k} \cdots N^{A_1} D_{A_1} \cdots D_{A_k} T  \in \Gamma(\ct^\Phi M[w-k])|_{\Sigma}\, ,\end{equation}
have transverse order $k$ for generic weights $w\notin\{
\frac{2k-d}2,
\frac{2k-1-d}2,
\ldots,\frac{k+1-d}2\}$~\cite{GPt}. An improvement of these operators acting on densities was also produced in~\cite[Theorem 4.16]{GPt}, where the set of weights for which the transverse order is less than $k$ was shrunk. In~\cite[Theorem 3.4]{B}, these improved operators were generalized to act on arbitrary tractors. In particular, for even $d$ and integers $0<J$ and $0<k<d/2$, it was shown that they give maps
\begin{equation}\label{normalops}
\delta_{J,k} : \Gamma(\ct^\Phi M[w]) \rightarrow \Gamma(\ct^\Phi M[w-k-J])\big|_{\Sigma}\, ,
\end{equation}
with transverse order $J+k$ also at the special weights $w = k-d/2$. Importantly, these existence proofs are both constructive and algorithmic.

\section{
Dirichlet-to-Neumann Tensors}\label{D2N}


Any  Poincar\'e--Einstein structure $(M_+,g^o)$
determines a conformal embedding 
$\Sigma\hookrightarrow (\overline{ M_+},\cc)$ of its boundary,
where $\cc$ is the conformal class of  metrics on  $M:=\overline{ M_+}$ determined by $g^o$. 
Moreover, this embedding must be umbilic~\cite{LeBrun,Goal}, so the corresponding trace-free second fundamental form~$\IIo=0$. This tensor is an example  
of a natural conformal hypersurface invariant (see below).
In fact, a slew of other  natural conformal hypersurface invariants 
are also forced to vanish by the Einstein condition; this underlies the conformal fundamental forms construction of~\cite{BGW}. 

\smallskip

 We next need to define natural conformal hypersurface invariants.
 First recall that a natural Riemannian invariant on a Riemannian manifold is any tensor-valued polynomial made from the metric, its inverse, the Riemann tensor, and covariant derivatives of the latter. Strictly, this is the class of ``even invariants'', meaning those that are unaffected by a change of orientation. It is clear that this type of invariant is all that will be required in our context.
 This definition extends in the obvious way to a Riemannian manifold equipped with any set of  scalar functions. Extending these to  natural conformal hypersurface invariants is discussed in detail in~\cite[Section 2.4]{CAG}. 
 Here we are interested in the same notion 
 with some restrictions.

 Consider  triples~$(M,\cg,\sigma)$, where~$\sigma$ is a defining density for 
a conformally embedded hypersurface~$\Sigma\hookrightarrow (M,\cc)$.  
 For any  triple~$(M,\cg,\sigma)$, now consider a representative pair $(g,s)$ for $\sigma$. 
 Then we can construct the  natural Riemannian invariants~$i(g,s)$
built from the metric $g$ and the scalar function $\hat s:=s/|\ext  s|_g$ (in some collar neighborhood of $\Sigma$ where the denominator is non-vanishing), and then restrict these to $\Sigma$. In turn we consider the special class of such  objects~$\big[g;i(g,s)|_\Sigma\big]\in \Gamma({\mathbb T}M)|_\Sigma$
for which
$$
\bar \Omega^{-w} \hh i(\Omega^2 g, \Omega s)|_\Sigma
=i(g,s)|_\Sigma\, 
\quad	\mbox{ and } \quad 
i(g,e^\varpi s)|_\Sigma=i(g,s)|_\Sigma\, ,
$$
where $0<\Omega\in  C^\infty M\ni \varpi$, $\bar \Omega =\Omega|_\Sigma$  and $w\in {\mathbb R}$. The equivalence class  $\big[g,i(g,s)|_\Sigma\big]$
defines a  (weight $w$)  {\it  natural conformal hypersurface invariant}. Upon dropping the conformal requirement imposed by the first equality displayed above, we employ the terminology {\it natural Riemannian hypersurface invariant}. 

We use the same 
language to also refer to the larger class of
natural Riemannian hypersurface invariants $i(g,s)$ that may only
obey the requirements above 
 when restricting to the subset of  triples~$(M,\bg,\sigma)$ such 
that ~$(M_+,\sigma^{-2}\bg)$ is Poincar\'e--Einstein. 
  An  example  of such a natural conformal hypersurface invariant for Poincar\'e--Einstein structures
 is~$\otop C_{\hat n (ab)}|_\Sigma\in \Gamma(\otop\!\odot^2\!\hh T^*M[-1])|_\Sigma$   (see~\cite{GoKo,BGW}). 
  For brevity, we often refer to such tensors as ``natural'' and list the relevant section spaces and underlying structures.
   In some contexts we additionally  require that the boundary~$(\Sigma,\bar \cc)$ is 
 conformally Einstein or conformally flat.

We employ the terminology {\it preinvariant} for any representative $i(g,s)$ of a natural conformal hypersurface invariant $I$ (see~\cite{CAG}).
Examples of natural conformal hypersurface invariants include the unit conormal $\hat n$ and trace-free second fundamental form~$\IIo$. Sample preinvariants for these are, respectively, 
$$
\ext \hat s=\ext\big(s/|\ext  s|_g\big)\, ,\quad
\big(\operatorname{Id}
-\hh
\ext \hat s\otimes g^{-1}(\ext\hat s,\pdot)
\big)\circ \nabla \ext \hat s \, .
$$

Consider a conformal hypersurface invariant $I$  that is represented by $i(g,s)$ and obeys
$$i(g+s^kh,s)|_\Sigma\neq i(g,s)|_\Sigma= i(g+s^{k+1}h',s)|_\Sigma\, ,\quad k\in {\mathbb Z}_{\geq 0}\, , $$
for some $h$ and any $h'\in \Gamma(\odot^2T^*M)$ such that $ g+s^kh,g+s^{k+1}h'$ are metrics on $M$.  
 The number~$k$  does not depend on  the  choice of $(g,s)$ used to construct $I$. Thus (in line with definition of Equation~\nn{TO})
 we then 
say that $I$ has {\it transverse order $k$}. 

\smallskip
In fact, Dirichlet-to-Neumann tensors are  natural  hypersurface invariants:
\begin{proposition}\label{Redford} 
Let $(M^d_+,g^o)$ be a Poincar\'e--Einstein structure. In the case that $d$ is odd, also assume that  the boundary $\Sigma$ is conformally flat.
Then the Dirichlet-to-Neumann tensor $\operatorname{DN}^{(d)}$ is a
natural  hypersurface invariant.
\end{proposition}
\begin{proof}
First recall that, given a Riemannian hypersurface embedding $\Sigma\hookrightarrow(M,g)$ and any defining function $s$, it is always possible to improve $s$  to $s_1=f s$ such that
$$
|ds_1|_g^2=1+{\mathcal O}(s^\ell)\, ,
$$
for any integer $\ell$, and the smooth function 
 $f$ is some local formula in terms of $s$ and $g$; see~\cite{CAG} (this is an asymptotic solution to  the well-known eikonal problem). To (arbitrarily) high order, the improved defining function $s_1$ gives the geodesic distance to the hypersurface.
 
%
%
%

Now pick $\bar g\in \cc_\Sigma$. This  determines the corresponding Graham--Lee compactified metric~$g_r$.
Then applying the above eikonal construction to $g_r$, 
starting with any defining function $s$, 
at least to sufficiently high order for our purposes, 
we obtain an asymptotic formula 
$s_{\rm GL}^{(\ell)}$
for the 
 Graham--Lee defining function $s_{\rm GL}$ as a natural formula in terms of $s$ and $g_r$ (this follows directly from the construction and uniqueness statement given in~\cite[Proposition 2.5]{CAG}). Moreover, by picking large enough $\ell$, we have that
$\hat s= s^{(\ell)}_{\rm GL}/|\ext s^{(\ell)}_{\rm GL}|_{g_r}$ equals the geodesic normal coordinate function $r$ to  any desired order.
Thus we can ignore higher order corrections in what follows.

Now recall that given a vector $x\in \Gamma(TM)$ and a tensor $T\in \Gamma(\odot^2T^* M)$, one has
$$
{\mathcal L}_x T_{ab} = x^c \nabla_c T_{ab} + (\nabla_a x^c) T_{cb}+
(\nabla_b x^c) T_{ac}
\, .$$ 
Note that the vector $\frac{\partial}{\partial r}$ determined by the normal form of the Graham--Lee compactified metric is natural because it solves the equation $g_r(\frac{\partial}{\partial r},\pdot)=\ext s_{\rm GL}$
and $s_{\rm GL}$ has a natural formula in terms of $s$.
Because $\hat s=r$, we have
 $\mathcal{L}_{\frac{\partial}{\partial r}} g_r=
2\nabla^{g_r} \ext \hat s$. Similarly, higher Lie derivatives with respect to the vector~$\frac{\partial}{\partial r}$ can also be expressed naturally in terms of $ \ext \hat s$, the metric~$g_r$, its inverse and covariant derivatives. 
Hence we have by now expressed such Lie derivatives of the Graham--Lee compactified metric $g_r$  as a natural  formula in~$s$ and the metric $g_r$.
The result follows upon remembering that for Poincar\'e--Einstein structures, the tensor $\operatorname{DN}^{(d)}$ is defined through $d-1$ Lie derivatives of the Graham--Lee compactified metric $g_r$ with respect to the vector $\frac{\partial}{\partial r}$.
\end{proof}

\color{black}

\subsection{$T$-curvatures}

Before proving our  main results we need the notion of a $T$-curvature and a result of~\cite{GPt}. Recall that the mean curvature $H^{g}$ behaves as follows under conformal rescalings
$$
H^{\Omega^2 g}= \Omega^{-1}\big(H^{g}+\delta_1 \log \Omega\big)\, ,
$$
where $\delta_1=\delta_R$ is the first of the sequence of operators defined in Equation~\nn{roughlyhalf} in the case where $w=0$ and~$\Phi$ is the trivial representation (and also in the scale $g$). 
Each {\it $T$-curvature} is a higher order generalization of the mean curvature to a natural Riemannian hypersurface invariant $T_k^g$ with conformal variation
\begin{equation}\label{skibidi}
T_k^{\Omega^2g}=\Omega^{-k} \big(T_k^g + N_k \log \Omega)\, ,
\end{equation}
where $N_k$ is any conformally invariant hypersurface operator with differential order $k$ and with leading symbol containing a term arising from an operator proportional to $\colon \nabla_{\hat n}^k\colon\hh$ for some $k\geq 2$.
The integer $k$ equals the transverse order of the  $T$-curvature $T_k$, so we shall call $T_k$ a {\it $T$-curvature of order $k$}.
For example, $T_1=H$ is the (up to an overall  non-zero coefficient)   first ~$T$-curvature and the scalar 
$$
T_2:=P_{\hat n \hat n}-\frac12 \,  H^2
 -\frac1{d-3} \bar J
$$
is also a $T$-curvature~\cite[Section 7.1]{GPt}.
Existence of $T$-curvatures of arbitrary order on any even-dimensional conformal manifold with boundary was established in~\cite[Theorem 4.16]{GPt}.
Moreover, in that work, the following result was proved:
\begin{proposition}[\cite{GPt} Proposition 6.15]\label{rehash}
Let $\Sigma \hookrightarrow (M^d,\cc)$ be a conformal hypersurface embedding with $d$ even.
Then, for any finite set of $T$-curvatures~$\{T_1^g,\ldots, T_k^g\}$,   there exists a metric  $g'\in  \cc$ such that
$$
T_i^{g'}=0\, , \quad\forall i\in\{1,\ldots, k\}\, .
$$ 
\end{proposition}
An example of how we plan to use the above proposition   is as follows: 
The usual definition of the dimension $d\geq 4$
{\it Fialkow tensor} is (for some choice of $g\in \cc$)
$$ 
F_{ab}:=P_{ab}^\top -\bar P_{ab} +H \IIo_{ab}  + \frac12 \bar g_{ab} H^2 \stackrel\Sigma= 
\frac1{d-3}\Big(
W_{\hat n a \hat n b} +
\IIo_{ac}\IIo^c{}_b -\frac1{2(d-2)}\bar g_{ab} |\mm\IIo\hh|^2\Big)\, .
$$
The second equality above is a standard hypersurface identity establishing that the Fialkow tensor defines a conformally invariant section of $\otop\!\odot^2 T^*M [0]|_\Sigma$. 
This equation can be equivalently rewritten  as
\begin{equation}\label{Rnanb}
R_{\hat n a\hat n b }^\top
\stackrel\Sigma=
\bar P_{ab}-\frac1{d-3}\hh \bar g_{ab} \bar J
 +(d-2) F_{ab}
-\IIo_{ac}\IIo^c{}_b
+\frac1{2(d-2)}\hh \bar g_{ab} |\mm \IIo\hh|^2
-H\IIo_{ab}
+\bar g_{ab}
 T_2
 \, .
\end{equation} 
Importantly both the trace-free Fialkow tensor and trace-free second fundamental form vanish for APE structures of sufficiently high order; they are the first two conformal fundamental forms discussed above
(note also that the trace $F_a{}^a=\frac1{2(d-2)}\hh |\mm\IIo\hh|^2$). Thus for $\APE{1}$ structures, the right hand side of Equation~\nn{Rnanb} reduces to the sum of an intrinsic tensor plus a $T$-curvature.
Hence Proposition~\ref{rehash} can be used to establish that there exists a metric in the conformal class~$\cc$ such that~$R_{\hat n a\hat n b }^\top|_\Sigma$ is expressible in terms of intrinsic quantities alone when the structure is Poincar\'e--Einstein. \color{black}
The  above decomposition into intrinsic curvatures, conformal fundamental forms, and $T$-curvatures applies generally to bulk curvatures and their derivatives, see~\cite{B1}
where a classification of certain hypersurface invariants was given (note that we in particular rely on Theorem 2.5 of that work which applies for $d$ even). It also underlies the strategy for the proof of Theorem~\ref{zero-invt}, where it is used to show that 
derivatives of ambient curvatures can be re-expressed in terms of intrinsic invariants, conformal fundamental forms and $T$-curvatures. The last of these  can be eliminated using Proposition~\ref{rehash}.

\medskip




In the case of Poincar\'e--Einstein structures there is a particularly simple $T$-curvature construction. This involves log densities given by $\log \tau$, where $\tau$ is a true scale.  These are
defined in~\cite{GW}, but the details are not important here.
The key point is that the {\it Laplace--Robin operator}~$I_\sigma. D$ acting on a log density returns a density of weight $-1$, according to
\begin{equation}\label{quantity}
I_\sigma. D \log \tau 
= \frac{d-2}\tau
I_\sigma.\hat D \tau + \frac{\si}{\tau^2} \, (\hat D \tau)^2\in \Gamma(\ce M[-1])\, .
\end{equation}
Then, following~\cite{GPt}, a $T$-curvature of order $k\leq  \lfloor \frac d2 \rfloor$ is given by 
$$
t_k^g:=\delta_R (I_\sigma. D)^{k-1}  \log \tau \big|_\Sigma\, .
$$
Notice that $t_1=\delta_R \log \tau|_\Sigma=-H$ because $\delta_R$ can here be replaced by $(d-2)^{-1} I_\sigma.D$ restricted to $\Sigma$.

When $d$ is even, a version of the above construction applies to $k>\frac d2$. This is based on  a polynomial weight continuation of the operator $\delta_R (I_\sigma. D)^{k-1}$  acting on densities of weight $w$; again see~\cite{GPt}. 
Let us now explain; 
for simplicity take  $k$  odd. Then one has
\begin{equation}\label{factors}
\delta_R (I_\sigma. D)^{k-1} = \hat P_k (d+2w-2k +2)(d+2w-2k+4) \cdots (d+2w -k-1)\, ,
\end{equation}
where the operator $\hat P_k$ is conformally invariant, even when  acting on weights $w$ such that any of the factors to its right vanish. 
It can be expressed in terms of bulk covariant derivatives, curvatures,  the scale $\sigma$, and depends polynomially on $w$. 

\begin{remark}
The 
operator $\hat P_{k}$ is proportional to the operator $$\delta_{J,k-J}:\Gamma(\ce M[k-J-\tfrac d2])\to \Gamma(\ce \Sigma[-J-\tfrac d2])\, ,$$  for some $J \in \{1,\ldots,\frac{k-1}2\}$ and odd $k>\frac d2$, constructed in~\cite{GPt} (see Equation~\eqref{normalops} above).
This corresponds to removing 
a factor $(d+2w-2k+2J)$ from the right hand side of Equation~\eqref{factors} and then continuing in $w$ to the critical value $w=k-J-\frac d2$.
\end{remark}

The action of $\hat P_k$  on a log density is defined by replacing  
appearances of $w$ by the operator $\underline w$ placed on the far right of any term. Note that $\underline w \log \tau = 1$.
Based on~\cite{GPt}, it is easy to  show that 
\begin{equation}\label{Pk}
t_k := \hat P_k \log \tau|_\Sigma
\end{equation}
defines a $T$-curvature for any odd $k>\frac d2$.
(A very similar  construction applies to even $k$ but is not important here.)
We now  characterize these $T$-curvatures for the Graham--Lee scale of a Poincar\'e--Einstein structure.
\begin{proposition}\label{fencepost}
Let $(M^d,\bg,\sigma)$ be Poincar\'e--Einstein. 
Then if $g_{r}$ is the Graham--Lee compactified metric 
for some $\bar g\in \cc_\Sigma$,
 then for any odd integer $k\leq \lfloor \frac d2 \rfloor$, 
$$
t^{g_{r}}_k =0\, .
$$
\end{proposition}

\begin{proof}
It is well known that the mean curvature of $\Sigma\hookrightarrow(M,g_r)$ vanishes so $t_1^{g_r}=0$.

Now let $\tau$ be any true scale and $g$ the corresponding compactified metric and call $s=\sigma/\tau$.
Then the quantity appearing in Equation~\eqref{quantity} (divided by $d-2$) can be expressed as
$$
\frac{I_\sigma.\hat D \tau}\tau + \frac1{d-2} \frac{\si}{\tau^2} \, (\hat D \tau)^2\stackrel g= -\frac{\Delta^{g} s+ J^{g}s}{d} - \frac{s}{d-2}  J^{g}\, .
$$
Here we have placed a $g$ above the equality sign to indicate that we have trivialized the density bundle~$\ce M[-1]$ using the choice of metric $g$ to obtain the right hand side of the above display. This notational device will appear again in the below.
Now, if we specialize the above formula to the Graham--Lee compactified metric $g_r$, then $s =r$ the geodesic distance to the boundary. Moreover the Graham--Lee metric has the normal form
$
g_r= dr^2 + h(r)
 $.
 We shall say that a polyhomogeneous expansion in the variable $r$ is even/odd 
  to order $m$ if there are no odd/even terms of order $m$ or less and there are no terms involving $r^k\log r$ for $k\leq m$. 
  For a Poincar\'e--Einstein structure,  $h(r)$  has an even expansion in the coordinate $r$ up to order $d-2$~\cite{FG}.
  Hence the above display has an  odd expansion in $r$ to  order $d-3$ (a slightly stronger statement is available when  $d$ is even, see the proof of Proposition~\ref{higher}). Both the Laplace--Robin and conformal-Robin operator  are  odd operators with respect to $r$ to some order. Indeed, acting on a density of weight $w$, employing $r$ as a coordinate in a collar neighborhood of~$\Sigma$, one has
  $$
I_\sigma. D\stackrel{g_r} = (d+2w-2)\partial_r -r \partial_r^2 + \operatorname{odd}(r)\, ,
  $$
  where (the operator) $\operatorname{odd}(r)$ stands 
  for terms that are odd 
   to order $d-3$. 
    Thus, acting at weight $w=0$ and in the Graham--Lee scale
 $$
 \delta_R 
 (I_\sigma. D)^{k-2}=
 (\partial_r + \operatorname{odd}(r))
 \underbrace{ \big(
  (d-2k+4) \partial_r - r \partial_r^2 + \operatorname{odd}(r)\big)
  \cdots
   \big(
  (d -2) \partial_r - r \partial_r^2 + \operatorname{odd}(r)\big)}_{k-2\:  \mbox{\small terms}}\,  .
 $$
 Hence 
 $\delta_R (I_\sigma. D)^{k-1}  \log \tau $ is odd to order  $d-2-k$ and hence clearly vanishes along $\Sigma$ when $k\leq \lfloor \frac d2\rfloor$ and $k$ is odd.
 \end{proof}

Specializing to even dimensions $d$ we have a stronger statement.

\begin{proposition}\label{higher}
Let $(M^d,\bg,\sigma)$ be Poincar\'e--Einstein and let $d$ be even.
Then if $g_r$ is the Graham--Lee compactified metric
for some $\bar g\in \cc_\Sigma$, then for any odd integer $k\leq d-1$, 
$$
t^{g_r}_k =0\, .
$$
\end{proposition}

\begin{proof}
The case when $k\leq \frac d2$ was  dealt with in 
Proposition~\ref{fencepost}. Hence we consider the $T$-curvatures defined in Equation~\eqref{Pk}.
The Laplace--Robin operator is given for some choice of $g\in\cc$ by
$$
I_\sigma. D = (d+2w-2) (\nabla_n + w \rho) - \sigma (\Delta + w J)\, ,
$$
where $\rho:=-\frac1d (\Delta \sigma + J \sigma)$.
Hence we see  that  $\hat P_k$ can be expressed as a sum of words of length $k$ built partly  from the four letters 
$$
\alpha:=
\nabla_n\, ,\quad \beta:=\rho\, ,\quad \gamma:=\sigma \Delta\, ,\quad
\delta :=  \sigma J\,  .
$$
We must also allow a fifth letter
$$
 \varepsilon:=\underline w
$$
that may appear as many as $2k-1$ times at the far right of any word.
It is given by the weight operator $\underline w$ which returns a factor of $w$, the weight of a density, when acting on such. Also, recall that~$\underline w \log \tau = 1$~\cite{GW}. 
As an example, $I_\sigma.D$ is expressed as a sum of the words  $ \alpha$, $\alpha \varepsilon$, $\beta \varepsilon$, $\beta \varepsilon^2$,~$\gamma$ and~$ \delta$.
Viewed as operators, in the Graham--Lee scale, each of the four letters $\alpha, \beta, \gamma, \delta$ is odd up to high order, or more precisely can be expressed as a sum of terms
$$
\partial_r\, ,\quad r\partial_r^2\, , \quad \operatorname{odd}(r)\, ,
$$
where $\operatorname{odd}(r)$ stands for terms odd in $r$ to order $d-3$. Also, in the Graham--Lee scale, one has $\log \tau =0$, $\nabla_n \log \tau =0$, $\Delta \log \tau=0$, while in general  $\underline w \log \tau = 1$ and  $\underline w^\ell \log \tau = 0$ for $\ell\geq 2$.
One can now apply the same parity argument as in proof of Proposition~\ref{fencepost} to establish that~$t_k$ vanishes for odd $k<d-1$. For the case $t_{d-1}$,  we must also carefully examine the highest $r$-derivatives of the terms $\operatorname{odd}(r)$. It not difficult to see that along~$\Sigma$ (where $r=0$) the highest order term is
$
\partial_r^{d-2} \operatorname{odd}(r)|_{r=0}
$.
Since this quantity is a scalar, a simple analysis of the operators appearing in $\alpha, \beta, \gamma, \delta$ shows that the above is proportional to 
$$
\operatorname{tr}_{h(r)}\partial_r^{d-1} h(r)\big|_{r=0}\, .
$$
This is precisely the boundary trace of the image of the Dirichlet-to-Neumann map. The latter is trace-free (see~\cite{FGbook}). 
\end{proof}

A uniqueness argument establishes  a vanishing result for odd order $T$-curvatures.
\begin{theorem}\label{skibido}
Let
$(M^d,\bg,\sigma)$ be Poincar\'e--Einstein and $d$ be even.
Then, if~$g_r$ is the Graham--Lee compactified metric
for some $\bar g\in \cc_\Sigma$,
any $T$-curvature of odd order $k\leq d-1$ vanishes.
\end{theorem}

\begin{proof}
We have  established that the odd order $T$-curvatures $t_k$ vanish, 
so it only remains to establish uniqueness of 
 odd order $T$-curvatures $T_{k\leq d-1}$ in the current setting.

On an even dimensional Riemannian manifold, every natural hypersurface invariant with transverse order $k$ may be expressed as a (partial) contraction polynomial in
$$\{\bar{g},\bar{g}^{-1}, \bar{\nabla}, \bar{R}, \hat{n}, \IIo, \otop \colon \nabla_{\hat n} \colon P, \ldots, \otop \colon \nabla_n^{k-2} \colon P, H, J|_{\Sigma}, \ldots, \colon \nabla_{\hat{n}}^{k-2} \colon J|_{\Sigma}\}\,;$$
see~\cite[Theorem 2.5]{B1}. 
Notably, on a Poincar\'e--Einstein manifold, when $k \leq d-2$, this family reduces to
$$\{\bar{g},\bar{g}^{-1}, \bar{\nabla}, \bar{R}, \hat{n}, H, J|_{\Sigma}, \ldots, \colon \nabla_{\hat{n}}^{k-2} \colon J|_{\Sigma}\}\,.$$
To see this, firstly note that conformal fundamental forms exist up to $\FFo{d-1}$ for embeddings in even dimensional conformal manifolds~\cite[Proposition 3.6]{B}. Moreover, normal derivatives of~$P$  may be expressed in terms of these conformal fundamental forms as well as terms from the list directly above; see~\cite[Corollary 3.3]{B1}.
However, on Poincar\'e--Einstein manifolds, all conformal fundamental forms up to and including $\FFo{d-1}$ (which has transverse order $d-2$) vanish~\cite[Theorem 1.8]{BGW1}. 
We wish to consider transverse order as high as $k=d-1$, and so must in principle also  include $\otop \colon \nabla_{\hat{n}}^{d-3} \colon P$ in the above family of terms.


By the transverse order requirement for  $T$-curvatures, 
the log term of Equation~\eqref{skibidi} must take the form
\begin{equation}\label{matchme}N_k \log \Omega = \Omega^{-1} \colon \nabla^k_{\hat{n}} \colon \hh\Omega + \ltots\, ,
\end{equation}
where ``ltots'' stands for terms involving only  lower transverse order operators acting on  $\Omega$. 
Also, it follows from the conformal transformation rule for the Schouten tensor $P$, that
$$J^{\Omega^2 g} 
\= J^g - \Omega^{-1} \colon \nabla_{\hat n}^2\hh \colon \Omega+\ltots{}\, .$$
And so, for $2 \leq k \leq d-2$, the only term in our family with the correct transformation law (matching Equation~\nn{matchme})
at leading transverse order is $\colon \nabla_{\hat n}^{k-2} \colon J|_{\Sigma}$. (Note that at order $k=d-1$, the  term $\otop \colon \nabla_{\hat n}^{d-3} \colon P$
is precluded from appearing in the leading transverse order term of our putative $T$-curvature by a combination of its
 its tensor structure and a homogeneity argument.)
This establishes that, on a Poincar\'e--Einstein manifold, any $T$-curvature $T_k$ has leading transverse-order term 
proportional to $\colon \nabla_{\hat n}^{k-2} \colon J|_{\Sigma}$.

Under constant conformal transformations $g \mapsto \lambda^2 g$, any $T$-curvature $T_k$ is homogeneous of weight $-k$. But $\bar{g}$, $\bar{g}^{-1}$, $\bar{\nabla}$, $\bar{R}$, and all even-order $T$-curvatures have \textit{even} homogeneity. Thus, homogeneity implies that each subleading term in any {\it odd} $T$-curvature can  be expressed in a form that  contains at least one (lower-order) odd $T$-curvature.
(Here we used a tensor plus weight argument to preclude terms involving $\hat n$.)

The proof is completed  by induction in the order of $T$-curvatures.
 The base case relies  
on the fact that any first $T$-curvature $T_1\propto H$ which vanished in the Graham--Lee scale.
Also,  we have just shown that distinct odd order $T$-curvatures differ only by terms that can be expressed in terms of lower order $T$-curvatures. Consulting  Proposition~\ref{higher} completes the proof.
\end{proof}

\subsection{Even dimensional bulk}

Here we focus on establishing Theorem~\ref{peachy} concerning even-dimensional Poincar\'e--Einstein structures.
The following theorem is key to the proof.

\begin{theorem} \label{zero-invt}
Let $\Sigma \hookrightarrow (M^d,\cg)$
be a conformal hypersurface embedding and  $(M^d,\cg,\sigma) \in \APE{k-2}$ with  $2 \leq k \leq d-1$ and $d$ even. Moreover let $I \in \Gamma({\mathbb T}\Sigma[w])$, with $w \in 2 \mathbb{Z}+1$, be a natural  conformal hypersurface invariant with transverse order $0 \leq m \leq k-1$. Then, $I = 0$.
\end{theorem}

\begin{proof}
 We  rely on the decomposition discussed above.
By definition, natural hypersurface invariants  
are built from polynomials in Riemann curvatures, unit conormals, their (covariant and possibly tangential) derivatives, and metric contractions thereof. Since we are studying sections $I$ of~${\mathbb T}\Sigma$, 
we may always deal with expressions such that all (undifferentiated) unit conormals are contracted into Riemann curvatures and their derivatives
(a pair of unit conormals can always be re-expressed as $\hat n_a \hat n_b  = g_{ab}|_\Sigma-\bar g_{ab}$ while no invariant proportional to an overall factor  of a single $\hat n_a$ can live in $\Gamma({\mathbb T}\Sigma)$).

\smallskip
 We now analyze when $I$ can have a given transverse order.
Consider first the tensor
\begin{equation}\label{Xfactor}
X_{ab}^{(\ell)}:=\hat n^c \hat n^d \hh\colon \nabla_{\hat n}^\ell\colon \hh R_{cadb}\big|_\Sigma \in \Gamma({\mathbb T}\Sigma)\, .
\end{equation}
Clearly $X_{ab}^{(\ell)}$ clearly has transverse order at most $\ell+2$; this  fits with  Equation~\nn{Rnanb} because~$T_2$ and the trace-free Fialkow tensor have transverse order~$2$. To see that the transverse order is exactly $\ell+2$, we study how $X^{(\ell)}_{ab}$ behaves upon replacing $g$ by $g+s^{q} h$. For that we  consider the linearization of the Riemann tensor around $g$ for the perturbation $g + s^q h$, namely
\begin{align}
R_{abcd}^{g+s^{{q}} h}\!\!-\!\!R_{abcd}^g&=-\tfrac12\big(\nabla^g_a \nabla^g_c (s^{q} h_{bd})
-\nabla^g_b \nabla^g_c (s^{q} h_{ad})
-\nabla^g_a \nabla^g_d (s^{q} h_{bc})
+\nabla^g_b \nabla^g_d (s^{q} h_{ac}) 
\big)+{\mathcal O}(s^{{q}-1})
\nonumber
\\[1mm]
&=-\tfrac{{q}({q}-1)}2 s^{{q}-2}\big(\hat n_a \hat n_c  h_{bd}
-\hat n_b \hat n_c h_{ad}
-\hat n_a \hat n_d  h_{bc}
+\hat n_b \hat n_d  h_{ac} \big)
+{\mathcal O}(s^{{q}-1})\, ,
\label{samisirked}
\end{align}
where the notation $\hat n$ has been recycled to denote any smooth extension  of $\hat n$ to $M$, similarly below we will employ the notation $\top$
to denote any extension of the corresponding projector to~$M$. 
In turn we have
 $$\top\Big(
\hat n^c \hat n^d  \hh\colon \nabla_{\hat n}^\ell\colon \hh\big(R_{cadb}^{g+s^{\ell+2} h} -
R_{cadb}^{g}\big)\Big) =
 - \tfrac {(\ell+2)!}2  h_{ab}^\top + {\mathcal O}(s)\, .
 $$
This establishes that $X_{ab}^{(\ell)}$ has transverse order $\ell+2$.  
Along similar lines, inspecting the (suitably manipulated) Codazzi and Gau\ss\  relations
\begin{align}
\label{cod}
R_{abc \hat n}^\top &\eqSig \bar{\nabla}_a \IIo_{bc} - \bar{\nabla}_b \IIo_{ac}+
\bar g_{bc} \bar \nabla_a H - \bar g_{ac} \bar \nabla_b H
\,,\\[2mm]
\label{ga}
R_{abcd}^\top &\eqSig \bar{R}_{abcd} - \IIo_{ac} \IIo_{bd} + \IIo_{ad} \IIo_{bc}\nonumber
\\&\qquad\qquad\!\!
-H(\bar g_{ac} \IIo_{bd}
-\bar g_{bc} \IIo_{ad}
-\bar g_{ad} \IIo_{ba}
+\bar g_{bd} \IIo_{ac}
)
-(\bar g_{ac} \bar g_{bd} - \bar g_{bc} \bar g_{cd})H^2
\,,
\end{align}
shows that 
$Y_{abc}^{(\ell)}:=\big( \hat n^d \hh\colon \nabla_{\hat n}^\ell\colon \hh R_{abcd}\big)^\top$
and
$Z_{abcd}^{(\ell)}:=\big(\colon \nabla_{\hat n}^\ell\colon \hh R_{abcd}\big)^\top$  both
have at most transverse order  $\ell+1$. Indeed, a similar linearized Riemann argument to the above establishes that their transverse orders are exactly $\ell+1$; the tensors $X$, $Y$, $Z$, and the second fundamental form~$\II$ are the basic atoms from which the invariant $I$ is produced via contractions and hypersurface derivatives. Equations~\nn{cod} and~\nn{ga} satisfy the same type of decomposition as outlined for Equation~\nn{Rnanb},
namely into intrinsic, vanishing-for-$\APE{k}$-structures ($k$ sufficiently high), and (hypersurface derivatives of) $T$-curvatures.

We next need to show that the above decomposition property holds upon  suitable  application 
of normal derivatives~$\colon\nabla_{\hat n}^\ell\colon$
required to construct $X$, $Y$ and $Z$.
This has been established in a slightly different context
in~\cite{B1}. The argument is as follows:
first we need to show that normal derivatives of the ambient Schouten tensor, or its trace, can be traded for 
conformal fundamental forms and~$T$-curvatures at the cost of lower transverse order terms. 
Therefore we  must express~$X$,~$Y$ and~$Z$ in these terms.
Indeed, a Bianchi identity 
maneuver can be employed to reduce expressions involving $Y$ and $Z$ (at leading transverse order) to ones involving~$X$.
Consider, for example (refering to~\cite{B1} for the $Z$ case), 
\begin{multline*}
Y_{abc}^{(\ell)}
= \big(\hat n^d\hat n^e
\colon\nabla_{\hat n}^{\ell-1}\colon
\nabla_e  
R_{abcd}\big)^\top
=-\big(\hat n^d\hat n^e  \colon\nabla_{\hat n}^{\ell-1}\colon
 (\nabla_a R_{becd}+\nabla_b R_{eacd})\big)^\top\\
 =-\big(
 \nabla_a^\top (\hat n^d \hat n^e
 \colon\nabla_{\hat n}^{\ell-1}\colon
R_{becd})
- \nabla_b^\top 
( \hat n^d \hat n^e
 \colon\nabla_{\hat n}^{\ell-1}\colon R_{eacd})\big)^\top+{\rm ltots}\\
 =-\bar \nabla_a X_{bc}^{(\ell-1)}
 +\bar \nabla_b X_{ac}^{(\ell-1)}+\ltots
 \, .
\end{multline*}
Here  ``ltots'' stands for lower transverse order tensors.

\smallskip

Now we can focus on $X^{(\ell)}$.
The case $\ell=1$ has the desired decomposition by virtue of Equation~\nn{Rnanb}. For higher $\ell$, we examine 
Equation~\nn{Xfactor}, which can be re-expressed as
\begin{align*}\label{Xfactor}
X_{ab}^{(\ell)}&=\big(\hat n^c \hat n^d \hh\colon \nabla_{\hat n}^\ell\colon \hh W_{cadb}
+
\colon \nabla_{\hat n}^\ell\colon P_{ab}\big)^\top
+ \bar g_{ab} \hat n^c \hat n^d  \hh\colon \nabla_{\hat n}^\ell\colon P_{cd}
+\ltots{}
\, .
\end{align*}
Next (remembering that identically $\nabla^e W_{caeb}=(d-3)(\nabla_c P_{ab}-\nabla_a P_{cb})$), notice that
\begin{align*}
\big(
\hat n^c \hat n^d \hh\colon \nabla_{\hat n}^\ell\colon \hh W_{cadb}\big)^\top
&=
\big(
\hat n^c (g^{de}- \bar g^{de}) \hh\colon \nabla_{\hat n}^{\ell-1}\colon \hh \nabla_e W_{cadb}
\big)^\top\\[1mm]
&=(d-3)\hh
\big(
\hat n^c \hh\colon \nabla_{\hat n}^{\ell-1}\colon \hh
(\nabla_c P_{ab}-\nabla_a P_{cb})\big)^\top
+\ltots{}
\\[1mm]
&=
(d-3)\hh
\big(
 \hh\colon \nabla_{\hat n}^{\ell}\colon \hh
 P_{ab}\big)^\top
+\ltots{}
\, ,
\end{align*}
so that
$$
X_{ab}^{(\ell)}
=(d-2)\big(
 \hh\colon \nabla_{\hat n}^{\ell}\colon \hh
 P_{ab}\big)^\top
 +\bar g_{ab} \hat n^c \hat n^d  \hh\colon \nabla_{\hat n}^\ell\colon P_{cd}
+\ltots{}\, .
$$
Now, the hypersurface trace-free part of the  first term on the right hand side above  can be re-expressed as 
an $(\ell+3)^{\rm th}$ conformal fundamental form plus lower transverse order terms---again see
~\cite[Corollary 3.3]{B1}.
The trace as well as the second term above can be written in terms of $T$-curvatures modulo lower transverse order terms. To see that, we may focus on the hypersurface trace
\begin{align*}
\bar g^{ab} X^{(\ell)}_{ab}&=
(d-2)(g^{ab}-\hat n^a \hat n^b) 
 \hh\colon \nabla_{\hat n}^{\ell}\colon \hh
 P_{ab}+(d-1)\hh\hat n^c \hat n^d  \hh\colon \nabla_{\hat n}^\ell\colon P_{cd}
+\ltots{}
\\
&=
(d-2) 
 \hh\colon \nabla_{\hat n}^{\ell}\colon \hh
 J
 + \hat n^a \hat n^b  \hh\colon \nabla_{\hat n}^\ell\colon P_{ab}
+\ltots{}\\
&=
(d-2) 
 \hh\colon \nabla_{\hat n}^{\ell}\colon \hh
 J
 + \hat n^a 
 g^{bc}
  \hh\colon \nabla_{\hat n}^{\ell-1}\colon \nabla_c P_{ab}
+\ltots{}
\\
&=
(d-1) 
 \hh\colon \nabla_{\hat n}^{\ell}\colon \hh J
 +\ltots{}
\, .
\end{align*}
In the above we used $\nabla^a P_{ab}=\nabla_b J$.
We have by  now shown,  at leading transverse order,  that  any tensor built from $X$, $Y$, $Z$, and $\II$ (as well as gradients thereof, and the hypersurface metric), can be expressed in terms of $H$, normal derivatives of~$J$, conformal fundamental forms, and possibly intrinsic tensors.
It remains to focus on normal derivatives of $J$.
 We know that $T$-curvatures are Riemannian hypersurface invariants built from~$X$, $Y$, $Z$, $\II$, covariant derivatives thereof, and  hypersurface metrics.  
So from the above display, it must be that normal derivatives of $J$ can be expressed in terms of~$T$-curvatures, conformal fundamental forms, and  intrinsic tensors.
 Hence $X$, $Y$, and~$Z$
themselves (and any tensor built therefrom) are also expressible (at leading order) in terms of conformal fundamental forms, $T$-curvatures, and possibly intrinsic tensors.
By descent in the transverse order, it follows that this statement holds to all orders.

Recall that   $X_{ab}^{(\ell)}$ has transverse order $\ell+2$ and $I$ has transverse order  at most $m=k-1$. So 
there exists an expression for $I$ such that for any appearance
 of $X$   (from above we have that~$Y$ and~$Z$ can be reduced to  terms involving only  $X$
 at leading order) one has that that $\ell$ is not greater than $k-3$. But for  $\APE{k-2}$ structures the $k{}^{\rm th}$ fundamental form (which has transverse order $k-1$), and all lower 
order 
 fundamental forms vanish. Hence, $I$ is expressible in terms of $T$-curvatures and intrinsic tensors alone. 
 But (see~\cite[Theorem 6.15]{GPt} or Proposition~\ref{rehash}) there exists a scale for which these~$T$-curvatures vanish. 
 Moreover, natural  intrinsic curvatures are composed of sums of products of the Riemann tensor, metrics as well as covariant derivatives and traces thereof, all of which have even homogeneity under constant metric rescalings. 
 Thus, because the weight of~$I$ is odd, there are no such intrinsic tensors available as these are taken to be (polynomially) built from $\bar g$, $\bar g^{-1}$, $\bar R$, $\bar\nabla$ and contractions thereof. 
This establishes that~$I$ vanishes for one choice of scale, and hence by its conformal invariance, $I=0$.

\end{proof}

\begin{remark}
Note that another approach to establishing this result is to examine the tensor structures that can appear as coefficients in the Fefferman--Graham expansion of an $\APE{k-2}$ metric; this approach might be used to write a more general result for $d$ odd.
\end{remark}

\bigskip

We will need  explicit formul\ae\
for Dirichlet-to-Neumann tensors in even dimensions in order to prove Theorem~\ref{peachy}.

\begin{lemma} \label{DN-formula-even}
Let $(M^d_+,g^o)$ be a Poincar\'e--Einstein structure with $d\geq 6 $ even. 
Then,
$$\operatorname{DN}^{(d)}_{ab} \propto \otop\hh  \colon \nabla_{\hat n}^{d-5} \colon\hh B_{ab}^{g_r}\,,$$
where the right hand side is evaluated on the Graham–Lee compactified metric $g_r$ used to define the left hand side.
\end{lemma}
\begin{proof}
By definition, the image of the Dirichlet-to-Neumann map of a metric representative $\bar{g} \in \cc_{\Sigma}$ is given by $(\mathcal{L}_{\frac{\partial}{\partial r}})^{d-1} g_r\big|_{\Sigma}$, where $g_r$ is the corresponding Graham--Lee compactified metric.
 Calling $n=\ext r$, $n^\sharp:= g_r^{-1}( n,\pdot)$, and using that $|n|=1$, it follows that
\begin{align*}
\operatorname{DN}^{(d)}_{ab} = (\mathcal{L}_{n^\sharp})^{d-1} g_{ab} \eqSig (\mathcal{L}_{n^\sharp})^{d-2} (\nabla_a n_b + \nabla_b n_a).
\end{align*}
Now evidently, the leading derivative term in the above is the same as that of
$$2 \nabla_{n^\sharp}^{d-2} \nabla_{(a} n_{b)}\,.$$
Moreover, the natural hypersurface invariant $\operatorname{DN}^{(d)}_{ab}-2 \nabla_{n^\sharp}^{d-2} \nabla_{(a} n_{b)}$
has  transverse order no larger than $d-2$. 
So from~\cite{B1}, it is expressible entirely in terms of
Riemannian invariants intrinsic to~$\Sigma$, 
conformal fundamental forms ranging  from $\IIo$ to $\FF{d-1}$,
and  $T$-curvatures up to order $d-2$. However, as $(M^d_+,g^o)$ is Poincar\'e--Einstein, it follows from~\cite{BGW1} that all conformal fundamental forms from $\IIo$ to $\FF{d-1}$ vanish. 
 Hence, by the hypersurface invariant decomposition of type~\eqref{Rnanb} established earlier  in this section, the difference in question can be expressed solely in terms of intrinsic invariants and $T$-curvatures. Now, because this difference has a definite \textit{odd} homogeneity of $3-d$ under constant rescalings of the metric $g \mapsto \lambda^2 g$, and because intrinsic invariants always have even homogeneity, it follows that every summand in any such decomposition must involve at least one occurence of an odd-order $T$-curvature. But, as $g$ is a Graham--Lee compactified metric, it follows from 
 Theorem~\ref{skibido} that all odd  order $T$-curvatures up to order $d-1$ inclusive vanish, and so 
$$\operatorname{DN}^{(d)}_{ab} = (\mathcal{L}_{n^\sharp})^{d-1} g_{ab}|_{\Sigma} = 2 \nabla_n^{d-2} \nabla_{(a} n_{b)}|_{\Sigma}\,.$$

Now, using the Ricci identity, we have the following identity:
\begin{align*}
\nabla_n \nabla_a n_b &= n^c \nabla_c \nabla_a \nabla_b r \\
&= n^c R_{cabd} n^d + n^c \nabla_a \nabla_b n_c \\
&= R_{nabn} - (\nabla_a n^c) (\nabla_b n_c) + \tfrac{1}{2} \nabla_a \nabla_b n^2 \\
&= R_{nabn} - (\nabla_a n^c)(\nabla_b n_c)\,,
\end{align*}
where the last identity follows because $n^2 = 1$. Now using this identity in the formula for $\operatorname{DN}^{(d)}$, and throwing away lower transverse
order terms
because they either involve (vanishing) odd order $T$-curvatures or conformal fundamental forms, we have that
$$\operatorname{DN}^{(d)}_{ab} = 2 \nabla_n^{d-3} R_{nabn}\,.$$
Now
applying the same reasoning as directly above to the decomposition of Riemann into its Weyl and Schouten tensor pieces, we may in turn express this as
$$\operatorname{DN}^{(d)}_{ab} = -2(d-2) \otop\hh  \colon \nabla_{\hat n}^{d-3}  \colon\hh P_{ab}\, .$$
Finally, using Equation~\eqref{Johanne} we have $ \otop B=\otop \nabla_n^2 P  $ modulo lower order terms (which again vanish  by similar reasoning to above), we have that
$$\operatorname{DN}^{(d)}_{ab} \propto \otop\hh  \colon \nabla_{\hat n}^{d-5}  \colon\hh B_{ab}\,,$$
thus completing the proof.
\end{proof}

We are now ready to prove  Theorem~\ref{peachy}.

\begin{proof}[Proof of Theorem~\ref{peachy}]
To unify the $d=6$ and $\geq 8$ formul\ae, let us denote
 $\delta_{0,1}:=\delta_R$.
 We shall
  begin by checking that $\bar{q}^* \otop \delta_{\tfrac{d-6}{2},\tfrac{d-4}{2}} W$ has the required leading transverse order and that the tractor~$ \delta_{\tfrac{d-6}{2},\tfrac{d-4}{2}} W$ has a projecting part with the correct tensor structure---this will establish that the leading transverse order term in~$\FFdn$ is as required by the theorem. (Note that by its  construction, $\FFdn$ is clearly a natural conformal hypersurface invariant.)
\color{black}

\medskip

From~\cite[Theorem 3.4]{B1} we have that $\delta_{\frac{d-6}{2},\frac{d-4}{2}}$ acting on weight $-2$ tractors has transverse order $d-5$, and thus 
$$\delta_{\frac{d-6}{2},\frac{d-4}{2}} W \propto \colon \nabla_{\hat n}^{d-5} \colon W + 
\ltots\, .
$$
The highest transverse order tensor component of the $W$-tractor is the Bach tensor, so
 $$(d-4)\hh W_{ABCD}\stackrel {\sss{\mathcal T}_g}={4}X_{[A} Z_{B]}^a X_{[C} Z_{D]}^b B_{ab}^g  + \ltots
\,.$$
(Here $g \in \cc$ is any choice of metric representative and we employ this choice for the remainder of this proof.)
Thus, we have that
\begin{equation}\label{iamabove}\otop\delta_{\frac{d-6}{2},\frac{d-4}{2}}W_{ABCD} \propto  \otop X_{[A} Z_{B]}^a X_{[C} Z_{D]}^b  \colon \nabla_{\hat n}^{d-5} \colon B_{ab}^g + \ltots{d-2}\,,\end{equation}
where $\ltots{d-2}$ are tensor-valued tractors with transverse order less than or equal to $d-2$.  
The $d\geq 4$ Bach tensor $B_{bd}=\big(\frac{1}{d-3}\nabla^a \nabla^c +P^{ac}\big)W_{abcd}$
has transverse order~$4$.
To see 
this we apply Equation~\nn{samisirked} to the  Weyl tensor  and find
\begin{multline*}
\!\!\!\!\!\hh W_{abcd}^{g+s^{\ell} h}\!\!-\!\!W_{abcd}^g=
-\tfrac{\ell(\ell-1)}{2(d-2)} 
s^{\ell-2}
\Big(\big[(d-3)\hat n_a \hat n_c -\bar g_{ac} \big]h_{bd}^{\mathring\top}\hh
-\big[(d-3)\hat n_b \hat n_c -\bar g_{bc} \big]h_{ad}^{\mathring\top}\quad
\\
\qquad\qquad
-\big[(d-3)\hat n_a \hat n_d -\bar g_{ad} \big]h_{bc}^{\mathring\top}
+\big[(d-3)\hat n_b \hat n_d -\bar g_{bd} \big]h_{ac}^{\mathring\top}
 \Big)
+{\mathcal O}(s^{\ell-1})\, ,
\end{multline*}
where 
$\bar g_{ab}$  is any extension of $g_{ab}-\hat n_a \hat n_b$ to $M$ and
$h_{ab}^{\mathring\top}$
any such extension of $\bar g_{ac} \bar g^{cd} h_{db}-\frac1{d-1} \bar g_{ab } \bar g^{cd} h_{cd}$. 
Hence for the Bach tensor we obtain
$$
B_{ab}^{g+s^{\ell} h}\!-B_{ab}^g=
-\tfrac{\ell(\ell-1)(\ell-2)
(\ell-3)}{2(d-2)} 
s^{\ell-4}
h^{\mathring\top}_{ab}
+{\mathcal O}(s^{\ell-3})\, .
$$
This establishes the claimed transverse order result.
In turn, it follows that the tractor in Equation~\nn{iamabove} has transverse order $d-1$. It is also now clear that, for generic structures, the 
coefficient of $\otop X^{\phantom a}_{[A} Z_{B]}^a X^{\phantom a}_{[C} Z_{D]}^b$
in $\otop \delta_{\frac{d-6}{2},\frac{d-4}{2}} W$ is non-zero.

Now suppose  
the projecting part of $\otop\delta_{\frac{d-6}{2},\frac{d-4}{2}}W_{ABCD}$
is not 
 proportional to $\otop X^{\phantom a}_{[A} Z_{B]}^a X^{\phantom a}_{[C} Z_{D]}^b$. In that case it must consist of terms appearing among those labeled $\ltots{d-2}$ in Equation~\eqref{iamabove}.
 But then the putative  projecting part of $\otop\delta_{\frac{d-6}{2},\frac{d-4}{2}}W_{ABCD}$ would have both odd weight and transverse order strictly less than $d-1$ and thus vanish by Theorem~\ref{zero-invt}.
 Thus we must now have
\begin{align*}
\hh\otop\hh\delta_{\frac{d-6}{2},\frac{d-4}{2}}W_{ABCD} 
\propto   \otop \hh X_{[A} Z_{B]}^a X_{[C} Z_{D]}^b 
\left(
 \colon \nabla_{\hat n}^{d-5} 
\colon B_{ab}^g + \ltots{d-2}\right)\,.
\end{align*}
So by construction, we have that
\begin{equation}\label{equiv}\bar{q}^* \otop \delta_{\frac{d-6}{2},\frac{d-4}{2}} \hh W \propto \otop \colon \nabla_{\hat n}^{d-5} \colon B^g + \ltots{d-2}  \in \Gamma(\otop\!\odot^2 T^*M[3-d]|_\Sigma)\, .\end{equation}

\smallskip

We next must show that $\FFdn$ is the unique (up to multiplication by a constant and natural) conformal hypersurface invariant of transverse order $d-1$ in $\Gamma(\otop \odot^2 T^* M[3-d]|_{\Sigma})$ determined by the Poincar\'e--Einstein structure. 
The $d=6$ case is handled in Theorem~\ref{blunt}
whose proof is given in Appendix~\ref{wedetestappendices}.
%
By similar considerations as those given there, the leading transverse order term  of $\FFdn$ for $d\geq 8$ is unique.
Now suppose that there exist two distinct conformal hypersurface invariants
$$\otop \colon \nabla_{\hat n}^{d-5} \colon B + Q \:\text{ and } \:\otop\colon \nabla_{\hat n}^{d-5} \colon B +  Q'\,,$$
where $Q\neq Q'$ are tensors with transverse order less than or equal to $d-2$. Necessarily $Q - Q'$ is conformally invariant, has transverse order at most $d-2$, and weight~$3-d$ (an odd integer)
so vanishes by dint of
Theorem~\ref{zero-invt}.   This establishes uniqueness of $\FFdn$.

Finally, we must check that $\FFdn \simeq \operatorname{DN}^{(d)}$.
From Lemma~\ref{DN-formula-even}, this amounts to checking that the terms labeled $\ltots{d-2}$ in Equation~\eqref{equiv} vanish when 
evaluated on a Graham--Lee compactified metric.
This is achieved by recycling the  argument used to prove Lemma~\ref{DN-formula-even}.
Because it has already been established that $\operatorname{DN}^{(d)}$ is the functional gradient of the renormalized volume---see \cite[Section 3.4]{deHaro},~\cite[below Theorem 2.2]{Anderson} and ~\cite[Theorem 1.3]{Albin})---we have completed the proof.
\end{proof}

\begin{remark}
As discussed above the tensor $\operatorname{DN}^{(d)}$ is a functional gradient along  a path of Poincar\'e--Einstein metrics which 
can be labelled by a path of boundary metrics,  so  its  boundary divergence must vanish,
$$
\bar \nabla^a\operatorname{DN}^{(d)}_{ab} =0
\, .
$$ 
In fact both $\bar \nabla^a\operatorname{DN}^{(d)}_{ab}$ and 
 $\bar{\nabla}^a \FFdn_{ab}$ 
 are conformally invariant. Theorem~\ref{peachy} therefore establishes the vanishing of the latter.
 \end{remark}

\subsection{Odd dimensional bulk}

We first need to characterize conformal hypersurface invariants for (asymptotically) Poincar\'e--Einstein structures with conformally flat boundaries. The following is an analog of Theorem~\ref{zero-invt}.
\begin{lemma} \label{lemmita}
Let $\Sigma \hookrightarrow (M^d,\cg)$
be a conformal hypersurface embedding where  $(M^d,\cg,\sigma) \in \APE{k-2}$ with  $2 \leq k \leq d-1$. Moreover, suppose that  the boundary $(\Sigma,\cc_{\Sigma})$ is   conformally flat. Then any natural  conformal hypersurface invariant  $I \in \Gamma({\mathbb T}\Sigma[w])$  with transverse order $1 \leq m \leq k-1$ vanishes.
\end{lemma}

\begin{proof}
First we choose a scale for which the boundary metric $\bar g= \bar \delta$ is flat. Then, the corresponding Fefferman--Graham expansion of the compactified metric $g_r$ of an  $\APE{k-2}$ structure is 
\begin{equation}\label{flatty}
g_r = {\ext }r^2 + \bar \delta + r^{k} Y
\end{equation}
for some smooth, rank two, symmetric tensor $Y$ which is  trace-free by virtue of the
Poincar\'e--Einstein condition.
The above expansion follows  directly from the treatment of conformally flat and conformally Einstein spaces in~\cite[Chapter 7]{FGbook}; see also Equation~\nn{hypexp}
 below.
 The statement now follows since the only local conformal invariants for a flat metric are those  built from the metric itself,  and that  the order~$r^k$ term cannot contribute to invariants of transverse order strictly  less than $k$. Vanishing in one scale implies the same for all other scales.
\end{proof}


We are now ready to prove Theorems~\ref{spliff} and~\ref{theoremono}.

\begin{proof}[Proof of Theorem~\ref{spliff}]
By Lemma~\ref{lemmita}, we only need consider  invariants of transverse order zero or four. Because the boundary conformal class is flat,
the only intrinsic  (and hence transverse order zero) conformal invariants
 of the correct homogeneity and tensor type 
 are powers of the metric and its inverse (see also Remark~\ref{whyamIhere}).
 No trace-free tensors can be made this way.
 Then an exhaustive search establishes
 that the only Riemannian hypersurface invariant of the desired  homogeneity, tensor type and transverse order four, for the metric of Equation~\eqref{flatty} (with $k=4$), is~$B^\top_{(ab)\circ}$.
 But the conformal transformation of the Bach tensor is
 $$
\Omega^2 B_{ab}^{\Omega^2 g}-B_{ab}^g=
(d-4)\big(2C_{\Upsilon (ab)}
-W_{\Upsilon ab \Upsilon}
\big)\, ,
$$
and the right hand side vanishes when evaluated on  the metric of Equation~\eqref{flatty}, so $B^\top_{(ab)\circ}$ is conformally invariant.
Finally, a calculation of the $B^\top_{(ab)\circ}$ in the metric of Equation~\eqref{flatty} shows that this is proportional to $Y$ along $\Sigma$, which establishes that 
$
\Vo^{\sss\rm DN^\flat}\simeq \operatorname{DN}^{(5)}
$.
%
%
%
%
%
%
%
\end{proof}

\begin{remark}\label{whyamIhere}
 Note that the only non-zero, transverse order zero,  $d=5$, 
natural, conformally invariant, sections of $\odot_\circ^2 T^*\Sigma[-2]$ are~$\bar{B}_{ab}$ and $\bar{W}_{(a|cde} \bar{W}_{b)\circ}{}^{cde}$.
Also, for $\APE{2}$ structures with a generic conformal class $\cc_\Sigma$, 
there are no conformally-invariant trans\-verse-order 2 or 4 tensors of this type. 
A simpler version of the exhaustive analysis of tensor structures and their conformal variations given 
in Appendix~\ref{wedetestappendices}
can be used to establish this.
\end{remark}


\begin{proof}[Proof of Theorem~\ref{theoremono}]
We must verify the 
 well-definedness of two definitions, 
 perform a 
  computation for $\operatorname{DN}^{(d)}$,
  and establish a proportionality statement. To start with, it is known (see~\cite{FGbook}) that a $d\geq 7$ Poincar\'e--Einstein metric~$g^o$ with conformally flat boundary has the Fefferman--Graham expansion
\begin{equation}\label{hypexp}
g_r:=r^2 g^o = {\ext }r^2 + \bar g - \bar P r^2 + \frac14  \bar P^{\odot2}\hh r^4+{\mathcal O}(r^{d-1})\, ,
\end{equation}
where $\bar P$ is the Schouten 
tensor of $\bar g$. Importantly the above expansion holds for {\it any}  $\bar g\in \cc_\Sigma$.
In particular, choosing (locally) $\bar g$ to be the flat metric $\bar\delta$, it follows  that the Weyl tensor of $g_r$ obeys
$$
W_{ab}{}^c{}_d = {\mathcal O}(r^{d-3})\, .
$$
In turn, for any $g\in \cc$ we have that the tractor curvature
$$
F_{ab}{}^A{}_B = W_{ab}{}^{cd}Z_c^A Z_{dB}^{\phantom{A}}
+C_{ab}{}^c (Z^A_c X_B-Z_{cB} X^A)
= 
\begin{pmatrix}
0&0&0\\
\mathcal{O}(r^{d-4})&
\mathcal{O}(r^{d-3})&0\\
0&\mathcal{O}(r^{d-4})&0
\end{pmatrix}\, ,
$$
so that, for any $1\leq k\leq d-4$ and $\sigma=[g;s]$, 
$$
\colon \nabla_{\hat n}^{k} \colon  F_{ab}{}^A{}_B\stackrel {\sss{\mathcal T}_g}=
\begin{pmatrix}
\mathcal{O}(s^{d-k-3})&\mathcal{O}(s^{d-k-2})&0\\[1mm]
\mathcal{O}(s^{d-k-4})&
\mathcal{O}(s^{d-k-3})&\mathcal{O}(s^{d-k-2})\\[1mm]
0&\mathcal{O}(s^{d-k-4})&\mathcal{O}(s^{d-k-3})
\end{pmatrix}\, .
$$
But  any failure of $\colon \nabla_{\hat n}^{d-4} \colon  F_{ab}{}^A{}_B$ to be conformally invariant necessarily involves terms proportional to the above display with $k\leq d-5$. These vanish along~$\Sigma$; to see this, one can re-express the tractor connection~\nn{trac-conn} in a matrix notation and then explicitly examine the failure of the above to transform correctly. 

Thus we now observe that 
$$
\hat n^a\colon \nabla_{\hat n}^{d-4} \colon  F_{ab}{}^C{}_D\stackrel\Sigma=
\begin{pmatrix}
0&0&0\ \\
\hat n^a\colon \nabla_{\hat n}^{d-4} \colon \hh 
C_{a b}{}^c&0&0\ \\
0&- 
\hat n^a\colon \nabla_{\hat n}^{d-4} \colon \hh 
C_{a bd}&0\
\end{pmatrix}\, .
$$
Note that $ C_{ n[bc]}=-\frac12 C_{bc n}=-
n^a \nabla_{[b} P_{c]a}$ so $C_{n[bc]}
-n_{[b}C_{c]nn}
$ 
(which restricts along $\Sigma$ to $C_{\hat n [bc]}^\top$)
vanishes to one higher order than its symmetric counterpart and thus
$$\big(\hat n^a\colon \nabla_{\hat n}^{d-4} \colon \hh 
C_{a [bc]}\big)^{\!\top}\big|_\Sigma=0\, .$$
 Since $ \hat n^a\big(\colon \nabla_{\hat n}^{d-4} \colon  
C_{a b c}\big)\big|_\Sigma$ appears as the projecting part of a tractor, and is symmetric, 
we then have that 
$$
\otop\circ
q^*\big(\big[\hat n^a \colon \nabla_{\hat n}^{d-4} \colon  F_{ab}{}^C{}_D\big]\big|_\Sigma\big)
 = 
 \otop\circ
  \big[
 \hat n^a\colon \nabla_{\hat n}^{d-4} \colon \hh 
C_{a b c}\big]\big|_\Sigma
$$
 is a well-defined element of $\Gamma(\otop \odot^2 T^*M[3-d]|_{\Sigma})$.
 In the case $d=5$, a short computation shows that, in the current setting, the right hand side above equals~$B^\top_{(ab)\circ}$.
 
\smallskip

Unlike in the even dimensional case, no parity argument ensures that the relevant Fefferman--Graham expansion coefficient defining $ \operatorname{DN}^{(d)}$ is covariant with respect to different choices of the boundary metric representative. We shall check that it is by hand.
For that we 
define a reference metric 
$$
\tilde g_r:= {\ext }r^2 + \bar g - \bar P r^2 + \frac14  \bar P^{\odot2}\hh r^4
$$
on a collar neighborhood of~$\Sigma$. Note that the  above is conformal to the hyperbolic metric~\cite{FGbook}.
Then we define a smooth tensor $X^{\bar g}$ in a collar of the boundary (and by dint of smoothness, along~$\Sigma$ itself) by
$$
r^{d-1} X^{\bar g} = g_r-\tilde g_r\, .
$$
Upon choosing a different metric $\bar g'\in \cc_\Sigma$, with respect to the new distance  function $r'=\Omega(r) r$ (suppressing the dependence of $\Omega(r)$ on boundary directions), the new compactified metric $g'$  then obeys (by virtue of Equation\nn{hypexp})
$$
g'_{r'}-\tilde g_{r'}=r'^{d-1}X^{\bar g'}\, .
$$
But, again relying on~\cite{FGbook},
$$
g'_{r'}-\tilde g_{r'}=\Omega(r)^2(g_r-\tilde g_r) + {\mathcal O}(r^d)\, .
$$
Thus we learn that 
$$
r'^{d-1}X^{\bar g'}=\Omega(r)^2 r^{d-1}X^{\bar g}+ {\mathcal O}(r^d)\, .
$$
So using $r'=\Omega(r) r$ and $0<\Omega(0)=\bar \Omega\in C^\infty \Sigma$, we have that  
$$X^{\bar g'}\stackrel\Sigma=\, \bar \Omega^{3-d} X^{\bar g}\, .$$
Clearly $(d-1)!\hh X^{\bar g}|_\Sigma = ({\mathcal L}_{\frac{\partial}{\partial r}})^{d-1} g_r\big|_\Sigma$. To show that $X|_\Sigma$ is trace-free, we first note that the Poincar\'e--Einstein condition implies that (see for example~\cite{Goal})
$$
n^2 + 2\rho \sigma = 1\, ,
$$
where $n:=\ext r$ and  $\rho := -(\nabla^a n_a + r J)/d$,
which in the scale $\sigma = [g_r;r]$ implies
\begin{equation}\label{livingsanstrace}
0=-d \rho=\nabla.n + r J^{g_r}\, .
\end{equation}
With impunity, we may once again compute in the boundary flat metric $\bar \delta$  scale for which
\begin{equation}\label{superflatty}
g_r = {\ext }r^2 + \bar \delta + r^{d-1} X^{\bar \delta}\, .
\end{equation}
Then simple computations show
$$
\nabla.n =\, \tfrac{d-1}2 r^{d-2}\tr_{\bar \delta} X^{\bar \delta} +{\mathcal O}(r^{d-1})
$$
and 
$$
r J^{g_r}=-\tfrac{d-2}{2}r^{d-2}\tr_{\bar \delta} X^{\bar \delta} 
 +{\mathcal O}(r^{d-1})\, .
$$
Hence the right hand side of Equation~\nn{livingsanstrace} equals $\tfrac12\tr_{\bar \delta} X^{\bar\delta}$. So $X^{\bar g}|_\Sigma$ must be trace-free for any choice of boundary scale, so indeed 
$\operatorname{DN}^{(d)} := ({\mathcal L}_{\frac{\partial}{\partial r}})^{d-1} g_r\big|_\Sigma$ defines an element of $\Gamma(\odot_\circ^2 T^*\Sigma[3-d])$.

It only remains to check that $\operatorname{DN}^{(d)}$ and $\FFdnf$ are proportional when evaluated on a Graham--Lee compactified metric. By their respective conformal invariance, we may perform this computation in the metric choice of Equation~\eqref{superflatty}. It is now easy to verify the stated proportionality result.


%

\end{proof}

\begin{remark}
Fefferman and Graham~\cite{FGbook} show that the expansion~\nn{hypexp} also applies when  the boundary class of metrics admits an Einstein metric
when written 
in the Fefferman--Graham coordinate $r$ corresponding to the boundary Einstein metric, and that these expansions are diffeomorphic if  there happen to be distinct Einstein metrics in the boundary conformal class. This implies the existence 
of a Dirichlet-to-Neumann map that is an invariant of the Poincar\'e--Einstein structure for the boundary-Einstein case.
It is likely that there are other such constructions for distinguished boundary conformal classes.


\end{remark}

\section*{Acknowledgements}
We thank Robin Graham for useful discussions.
S.B. acknowledges the support of the Czech Science
Foundation (GACR)  grant GA22-00091S and the Operational Programme Research Development and Education Project No. \!CZ.02.01.01/00/22-010/0007541.
A.R.G. and A.W. 
 acknowledge support from the Royal Society of New Zealand via Marsden Grants  19-UOA-008 and 24-UOA-005. 
J.K.  acknowledges funding received from the Norwegian Financial Mechanism 2014-2021, project registration number UMO-2019/34/H/ST1/00636.
A.W.~was also supported by  Simons Foundation Collaboration Grant for Mathematicians ID 686131, and  thanks the University of Auckland for warm hospitality.

\appendix

\section{Proof of Theorem~\ref{blunt}} \label{wedetestappendices}

The uniqueness part of the \color{black} proof is by exhaustion. The key details are as follows. Given an embedded hypersurface $\Sigma \hookrightarrow (M,g)$, any natural diffeomorphism invariant along $\Sigma$ can be fully described by 
  contractions of the conormal $\hat{n}$, its tangential derivatives, the metric and its inverse, and the bulk curvature~$R^g$ along with its derivatives $\nabla^k R^g$. Note the embedding $\Sigma \hookrightarrow (M_+,g^o)$ for  Poincar\'e--Einstein structures is necessarily umbilic~\cite{LeBrun,Goal}.
 Hence tangential derivatives of~$\hat{n}$ only produce the mean curvature, its hypersurface derivatives and the induced metric. The latter can be expressed in terms of the unit conormal and ambient metric.
So, in a choice of metric representative $g \in \cc$, any natural conformal hypersurface invariant of $\Sigma \hookrightarrow (M,\cc)$ is expressible by such a diffeomorphism invariant.
This is established in~\cite[Proposition 2.2]{B1} (which is based on~\cite[Proposition 2.7]{CAG}). Thus  we construct all possible candidate diffeomorphism invariants
and then show that only a single combination of these produces the desired, transverse order $5$, conformal hypersurface invariant
in  $\Gamma(\odot^2_\circ T^*\Sigma[-3])$.
The homogeneity under constant metric rescalings of the above set of ``letters'' is tabulated below:
$$
 \begin{array}{||c| c||} 
 \hline
 \mbox{Ingredient}^{\phantom b^{\phantom c}}\!\!\! & \mbox{Weight}\\ [0.1ex] 
 \hline\hline
 g^{ab} & -2   \\ 
 \hat{n}_a & 1 \\
H^g & -1  \\
\bar{\nabla}_a & 0  \\
\nabla_a & 0  \\
R_{abcd} & 2  \\ [1ex] 
 \hline
\end{array}
$$

\noindent
We note that when building a rank-$2$ 
trace-free tensor with conformal weight $-3$ from the above letters, only the inverse metric can reduce tensor rank. Hence, to find all words with the aforementioned properties,  we are faced with 
a non-negative, integer-valued linear algebra problem 
whose solutions are listed below:
$$\bar{\nabla}^2 H^3, \;
 g^{-1} H^3 R, \;
 (g^{-1})^2 \hat{n}^2 H^3 R, \;
 (g^{-1})^2 \hat n (\bar{\nabla} H^2) R, \;
 (g^{-1})^2 \hat n H^2 \nabla R, \;
 (g^{-1})^2 (\bar \nabla^2 H) R, \;
 g^{-1} \bar{\nabla}^4 H, 
 $$
 $$
 (g^{-1})^3 \hat n^2 (\bar \nabla^2 H) R, \;
 (g^{-1})^3 H R^2, \;
 (g^{-1})^4 \hat{n}^2 H R^2, \;
 (g^{-1})^2 (\bar \nabla H) \nabla R, \;
 (g^{-1})^3 \hat{n}^2 (\bar \nabla H) \nabla R, \;
$$
$$
 (g^{-1})^2 H \nabla^2 R,\;
 (g^{-1})^3 \hat n H \nabla^2 R, \;
 (g^{-1})^4 \hat n \nabla R^2, 
  (g^{-1})^3 \hat n \nabla^3 R\,.
  $$
Note that the condition $\hat n_a \hat n^a=1$ gives an upper bound on the number of appearances of $\hat n$
conditioned on the other letters appearing. Also, since we are searching for rank-two, trace-free covariant tensors, all inverse metrics $g^{-1}$ must be completely contracted.

The number of independent rank two tensors built 
from linear combinations of
the above list of words reduces significantly
 because  $(M^6_+,g^o)$ is Poincar\'e--Einstein, 
which allows these tensors to be re-expressed in terms of $\bar{g}$, $\bar{g}^{-1}$, $\bar{\nabla}$, $\bar{R}$, $J$, $\nabla_{\hat n} J$, $H$, and $\otop \nabla_{\hat n} B$. The list of all  weight $-3$, rank-$2$ combinations of these letters is given by
$$
\otop \nabla_{\hat n} B, \;
\bar{\nabla}^2 \nabla_{\hat n} J, \;
 \bar{g}^{-1} \bar{R} \nabla_{\hat n} J, \;
 \bar{\nabla}^2 H J, \;
 \bar{g}^{-1} H J \bar{R},
 $$
 $$
 \bar{g}^{-1} \bar{\nabla}^4 H, \;
 (\bar{g}^{-1})^2 \bar{\nabla} H \bar{R}, \;
 (\bar{g}^{-1})^3 H \bar{R}^2, \;
 \bar{\nabla}^2 H^3, \;
 \bar{g}^{-1} H^3 \bar{R}\, .
$$
These yield $24$ linearly independent symmetrized and trace-free tensors  of the correct weight and transverse orders:
$$
{\odot^2_\circ} 
\Big\{
\top\nabla_{\hat n} B_{ab}, \;
 \bar{P}_{cd} \nabla_{\hat n} J, \;
 \bar{\nabla}_c \bar{\nabla}_d \nabla_{\hat n} J, \;
 (\bar{\nabla}_c H) \bar{\nabla}_d J, \;
 H J \bar{P}_{cd}, \;
 H \bar{\nabla}_c \bar{\nabla}_d J, \;
 \bar{W}^{acdb} \bar{\nabla}_a \bar{\nabla}_b H, \;
 $$
 $$
 \qquad\:
 \bar{C}_{acd} \bar{\nabla}^a H,
  H \bar{B}_{cd}, \;
 H \bar{W}_c{}^{eab} \bar{W}_{dabe}, \;
 H \bar{W}_c{}^{abe} \bar{W}_{dabe}, \;
 H \bar{W}_c{}^{bae} \bar{W}_{dabe}, \;
 H \bar{P}_{ca} \bar{P}^a_d, \;
   \phantom{\Big\}}
  $$
  $$
  \qquad\:\:
  H \bar{P}^{ab} \bar{W}_{acdb}, \;
 H \bar{J} \bar{P}_{cd},\;
H^3 \bar{P}_{cd}, 
  \bar{J} \bar{\nabla}_c \bar{\nabla}_d H, \;
 (\bar{\nabla} H_c) \bar{\nabla}_d \bar{J}, \;
 H \bar{\nabla}_c \bar{\nabla}_d \bar{J}, \;
\bar{\nabla}_c \bar{\nabla}_d \bar{\Delta} H, \;
     \phantom{\Big\}}
 $$
 $$
 \bar{P}^a_c \bar{\nabla}_a \bar{\nabla}_d H, \;
  H^2 \bar{\nabla}_c \bar{\nabla}_d H, \;
 H (\bar{\nabla}_c H) (\bar{\nabla}_d H), \;
 \bar{P}_{cd} \bar{\Delta} H\Big\}\,.
$$
Requiring that the conformal variation of
an arbitrary linear  combination of the above basis elements  vanishes, we find a  unique combination
for which the  coefficient of the (only) transverse order five tensor $\otop \nabla_{\hat n} B$ is unity. In a choice of metric $g \in \cc$ this is
$$
\otop \nabla_{\hat n} B_{ab} - 4 \bar{C}_{c(ab)} \bar \nabla^c H + 4 H \bar{B}_{ab}\,.
$$
The Poincar\'e--Einstein condition in six dimensions gives  $2 \bar{B} = \otop B|_{\Sigma}$~\cite{BGW}, and in turn the expression quoted in the theorem. The conformal  property ensures that the tensor defines an invariant of the Poincar\'e--Einstein structure.

Finally, when evaluated on the Graham--Lee compactified metric  $g_r$ associated with a boundary choice of metric representative $\bar{g} \in \cc_{\Sigma}$, we must show that ${\VIo^{\sss\rm DN}}$ agrees with $\operatorname{DN}^{(6)}$. It has already been established in Lemma~\ref{DN-formula-even} that $\operatorname{DN}^{(6)} \propto \otop \nabla_{\hat{n}} B$. Because  $g_r$ is a  Graham-Lee compactified metric, it follows by virtue of Theorem~\ref{skibido} that the first, third, and fifth $T$-curvatures then vanish. In particular, the first $T$-curvature is the mean curvature, so it trivially follows that in this metric representative,
$${\VIo^{\sss\rm DN}} = \otop \nabla_{\hat n} B\,.$$
The theorem follows.
\color{black}
\hfill$\square$

\end{document}